\documentclass[10pt]{article}
\usepackage[a4paper,margin=1in]{geometry}

\usepackage{titlesec}
\usepackage{CJKutf8}%Chinese
\usepackage{bm}
\usepackage{amssymb,version}
\usepackage{color}
\usepackage{verbatim}
\usepackage{fancyvrb}
\usepackage{mathtools}
\usepackage{enumerate}
\usepackage{lipsum}
\usepackage{multirow}

\usepackage{graphics}% Figure
\usepackage{epstopdf}
\usepackage{subcaption} 
\usepackage{subfig}% Subfigure
\usepackage{svg}
\usepackage{amsmath}
\usepackage{float}
\usepackage{placeins}

\allowdisplaybreaks[4]
\usepackage{mathrsfs}
\numberwithin{equation}{section}
\makeatletter
\let\oldref\ref
\renewcommand{\ref}[1]{\normalfont\oldref{#1}}
\renewcommand{\arraystretch}{1.3}
\makeatother

\newcommand{\norm}[1]{\lVert #1 \rVert}

\usepackage{amsthm}
\newtheorem{lemma}{Lemma}[section]
\newtheorem{theorem}{Theorem}[section]
\newtheorem{corollary}{Corollary}[section]
\newtheorem{remark}{Remark}[section]
\allowdisplaybreaks
\makeatletter
\def\env@cases{%
	\let\@ifnextchar\new@ifnextchar
	\left\lbrace
	\def\arraystretch{1.5}  % 统一设置为1.5倍行距
	\array{@{}l@{\quad}l@{}}}
\makeatother

\begin{document}
	
	\title{THE MAC SCHEME FOR LINEAR ELASTICITY IN DISPLACEMENT-STRESS FORMULATION ON NON-UNIFORM STAGGERED GRIDS
	\thanks{This work is supported in part by the National Natural Science Foundation of China (Grant Nos.12131014).}}
	
	\author{ Hongxing Rui \thanks{Corresponding author,School of Mathematics, Shandong University, Jinan, Shandong, 250100, P.R. China. Email: hxrui@sdu.edu.cn}.\and Weijie Wang \thanks{Shandong University Mathematical Research Center, Shandong University, Jinan, Shandong, 250100, P.R. China. Email: 202420960@sdu.edu.cn}.}
	\maketitle
	\begin{abstract}
		A marker-and-cell (MAC) finite difference method is developed for solving the two dimensional and three dimensional linear elasticity in the displacement–stress formulation (MAC-E) on staggered grids.  The method employs a staggered grid arrangement, where the displacement components are approximated on the midpoints of cell edges, the normal stresses are defined at the cell centers, and the shear stresses are defined at the grid points.This structure ensures local conservation properties and avoids spurious oscillations in stress approximation. A rigorous mathematical analysis is presented, establishing the stability of the scheme and proving the second-order $L^2$-error superconvergence for both displacement and stress. The proposed method is locking-free with respect to the Lamé constant $\lambda\in (0,\infty)$, making it suitable for both compressible and nearly incompressible elastic materials. Numerical experiments demonstrate the efficiency and robustness of the MAC-E finite difference scheme, and the computed results show excellent agreement with the theoretical predictions.
	\end{abstract}
	\section{Introduction}
	The analysis of the deformation of compressible and nearly incompressible linear elastic solids in the stress--displacement formulation has wide applications in many important fields. By solving linear elasticity problems to obtain stress and displacement, one can predict the stiffness and stability of structures, providing valuable guidance for engineering design, structural optimization, and safety evaluation.
	
	For several decades, the stress--displacement formulation has been one of the standard frameworks for linear elasticity, and the finite element method has been widely used within this setting. Nevertheless, constructing stable discretizations for this formulation remains nontrivial. In contrast to the mixed formulation of the Poisson problem (see, for instance, \cite{brezzi2012mixed}), the stress field in linear elasticity is additionally subject to a symmetry constraint, which significantly complicates the design of finite elements. A breakthrough was achieved by Arnold and Winther \cite{arnold2002mixed}, who proposed the first family of mixed finite elements based on polynomial shape functions for the stress variable. Since then, many stable mixed finite element (MFE) schemes have been developed (see \cite{adams2004mixed, arnold2008finite, arnold2003nonconforming, hu2008lower, man2009lower, arnold2007mixed, boffi2009reduced, cockburn2010new, hu2015family, hu2015new}).
	
	Because these elements rely on a local commuting property, they typically involve a large number of degrees of freedom, which makes their implementation rather involved. Moreover, when the material is nearly incompressible, the performance of many schemes deteriorates severely due to locking. Besides stability, a growing line of research focuses on preserving the intrinsic structure of the underlying PDEs. For example, in incompressible flow problems, recent work by Charnyi et al.\ \cite{charnyi2017conservation} demonstrates that different Galerkin formulations of the Navier--Stokes equations may lose important invariants—such as kinetic energy, momentum, angular momentum, helicity, and enstrophy—when the divergence constraint is imposed only weakly. They proposed the EMAC formulation, which restores these conservation properties at the discrete level. Such findings highlight the broader importance of designing structure-preserving discretizations, a principle that also motivates the development of alternative approaches in computational fluid mechanics and elasticity.
	
	In contrast to the structure-preserving approaches based on mixed finite element frameworks discussed above, the MAC (Marker-and-Cell) scheme, which is built on staggered grids, has long been widely used in computational fluid mechanics due to its conceptual simplicity, ease of implementation, and computational efficiency. The MAC scheme was originally proposed by Lebedev \cite{dong2023second} and Daly et al.\ \cite{welch1965mac} in the 1960s, and it has been extensively used in engineering practice and serves as the foundation for many flow simulation software packages \cite{nicolaides1992analysis}. The MAC scheme enforces the incompressibility constraint of the velocity field pointwise, ensuring that the divergence-free condition is satisfied at the discrete level. Moreover, it has been proven that the MAC scheme locally conserves mass, momentum, and kinetic energy \cite{perot2000conservation, perot2011discrete}. Until 2017, Rui and Li \cite{rui2017stability} established the second-order superconvergence of the velocity and pressure in the discrete $H^1$ and discrete $L^2$ norms, respectively, thereby resolving this problem. In \cite{rui2018locking}, a Poisson locking-free linear elasticity method based on the non-stress formulation was proposed.
	
	In this paper, we introduce the MAC scheme for the stress--displacement formulation of the linear elasticity equations on staggered non-uniform grids. In the proposed framework, the normal stresses are defined at the cell centers, the shear stresses are located at the grid points, and the displacements are defined at the midpoints of cell edges. Based on this arrangement, an efficient finite difference discretization method is constructed. The proposed method effectively resolves the locking phenomenon under strong symmetry conditions, preserves the local conservation property, and achieves stable and convergent results on both uniform and nonuniform grids. It thus provides a simple yet efficient alternative to traditional finite element methods. For the scheme, we first establish the LBB (Ladyzhenskaya--Babuška--Brezzi) condition and prove the stability. Subsequently, inspired by the analytical techniques in \cite{rui2017stability} for the steady Stokes equations, we prove the second order superconvergence properties. It is found that the convergence orders of the displacement and the tangential stress are independent of $\lambda$. Finally, several numerical experiments are conducted to verify the accuracy and convergence rates.
	
	The paper is organized as follows. In Section~2, we give the problem formulation and introduce some notations. In Section~3, we present the MAC-E scheme for linear elasticity, establish the LBB condition, and provide the stability results. In Section~4, we present the superconvergence analysis for the MAC-E scheme. In Section~5, we briefly explain how to extend the analysis to three-dimensional problems. In Section~6, several numerical experiments using the MAC-E scheme are carried out. Throughout the paper, we use $C$, with or without subscripts, to denote a generic positive constant, which may take different values at different occurrences.

	\section{The Problem and Some Preliminaries}
		Let $\Omega \subset R^d,d=2,3$ be a polygonal domain. We consider the following linear elasticity in displacement-stress formulation
		
		\begin{equation}
			\begin{dcases}
				\begin{aligned}
					A\underline{\sigma}&=\varepsilon(\boldsymbol{u})\ &in \,\ \Omega,\\
					\mathrm{div}\,\underline{\sigma}&=\boldsymbol{f} &in \,\ \Omega,
				\end{aligned}
			\end{dcases}
			\label{initial equations}
		\end{equation}
		with boundary condition 
			\begin{equation}
				u=0,on\; \partial \Omega.
			\end{equation}		
		Here A is the compliance tensor determined by material parameters of the elastic medium.
		In a homogeneous isotropic elastic medium, A has the form
		\begin{align*}
			A\underline{\sigma}=\frac{1}{2\mu}(\underline{\sigma}-\frac{\lambda}{d\lambda+2\mu}tr(\underline{\sigma})\underline{I}),
		\end{align*}
		$\lambda$ and $\mu$ are positive constant, called the Lame parameters, $tr(G)$ is the trace of function $G$, and $\underline{I}$ is the identity matrix.
		
			We consider the MAC scheme for the elasticity problem and called it MAC-E. The partitions and notations we use as follows.
			
			First, the two dimensional domain $\Omega$ is partitioned by $\mathcal{T}^1_h=\delta_x \times \delta_y$, where
		\[
		\begin{split}
			\delta_x &: 0 = x_0 < x_1 < \cdots < x_{n_x - 1} < x_{n_x} = a, \\
			\delta_y &: 0 = y_0 < y_1 < \cdots < y_{n_y - 1} < y_{n_y} = b.
		\end{split}
		\]		
		For possible integers $i, j$, $1 \leq i \leq n_x-1$, $1 \leq j \leq n_y-1$, define
		\begin{align*}
			&x_{i + 1/2} = \frac{x_i + x_{i + 1}}{2}, \quad h_{i + 1/2} = x_{i + 1} - x_i, \quad h = \max_i \{h_{i + 1/2}\}, \\
			&h_i = x_{i + 1/2} - x_{i - 1/2} = \frac{h_{i + 1/2} + h_{i - 1/2}}{2}, \\
			&y_{j + 1/2} = \frac{y_j + y_{j + 1}}{2}, \quad l_{j + 1/2} = y_{j + 1} - y_j, \quad l = \max_j \{l_{j + 1/2}\}, \\
			&l_j = y_{j + 1/2} - y_{j - 1/2} = \frac{l_{j + 1/2} + l_{j - 1/2}}{2}, \\
			&\Omega_{i + 1/2, j + 1/2} = (x_i, x_{i + 1}) \times (y_j, y_{j + 1}), \\
			&\Omega_{i,j} = (x_{i-1/2}, x_{i + 1/2}) \times (y_{j-1/2}, y_{j + 1/2}), 
		\end{align*}
		Furthermore, we define $h_0 = \frac{h_{1/2}}{2}$, $h_{n_x} = \frac{h_{n_x - 1/2}}{2}$, $l_0 = \frac{l_{1/2}}{2}$, $l_{n_y} = \frac{l_{n_y - 1/2}}{2}$. 
		
		We suppose the partition is regular, which means there is a positive constant $C_0$ such that
		\begin{align}
			\min_{i,j} \{h_{i + 1/2}, l_{j + 1/2}\} \geq C_0 \max_{i,j} \{h_{i + 1/2}, l_{j + 1/2}\}.\label{grid conditon}
		\end{align}
		
		For simplicity we also set $c_{l,m} = (x_l, y_m)$ where $l(m)$ may take values $i, i + 1/2(j, j+1/2)$ for integer $i(j)$. For a function $\phi(x, y)$, let $\phi_{l,m}$ denote $\phi(x_l, y_m)$. For discrete functions, we define
		\begin{equation}
			\begin{dcases}
				{[d_x \phi]}_{i + 1/2, m} = \frac{\phi_{i + 1, m} - \phi_{i, m}}{h_{i + 1/2}}, & {[D_y \phi]}_{l, j + 1} = \frac{\phi_{l, j + 1/2} - \phi_{l, j - 1/2}}{l_{j + 1}}, \\
				{[D_x \phi]}_{i, m} = \frac{\phi_{i + 1/2, m} - \phi_{i - 1/2, m}}{h_i}, & {[d_y \phi]}_{l, j + 1/2} = \frac{\phi_{l, j + 1} - \phi_{l, j}}{l_{j + 1/2}}.
			\end{dcases} 
			\label{chafenfuhao}
		\end{equation}
		On the boundary and interface of $\Omega$, define
		\begin{equation}
			\begin{dcases}
				{[D_x \phi]}_{0, m} = \frac{\phi_{1/2, m} - \phi_{0, m}}{h_0}, & {[D_x \phi]}_{n_x, m} = \frac{\phi_{n_x, m} - \phi_{n_x - 1/2, m}}{h_{n_x}}, \\
				{[D_y \phi]}_{l, 0} = \frac{\phi_{l, 1/2} - \phi_{l, 0}}{l_0}, & {[D_y \phi]}_{l, n_y} = \frac{\phi_{l, n_y} - \phi_{l, n_y - 1/2}}{l_{n_y}}.
			\end{dcases} 
		\end{equation}

		For the functions $\phi$ and $\theta$, we define some discrete $l^2$ inner products and norms as follows. 
		\begin{align*}
			&(\phi, \theta)_{ M} \equiv \sum_{i=0}^{n_x - 1} \sum_{j=0}^{n_y - 1} h_{i + 1/2} l_{j + 1/2} \phi_{i + 1/2, j + 1/2} \theta_{i + 1/2, j + 1/2},\\
			&(\phi, \theta)_{ T} \equiv \sum_{i=0}^{n_x} \sum_{j=0}^{n_y} h_i l_j \phi_{i, j} \theta_{i, j},\\
			&(\phi, \theta)_{TM} \equiv \sum_{i=1}^{n_x - 1} \sum_{j=0}^{n_y - 1} h_i l_{j + 1/2} \phi_{i, j + 1/2} \theta_{i, j + 1/2},\\
			&(\phi, \theta)_{MT} \equiv \sum_{i=0}^{n_x - 1} \sum_{j=1}^{n_y - 1} h_{i + 1/2} l_j \phi_{i + 1/2, j} \theta_{i + 1/2, j},\\
			&\norm{\phi}_\xi^2 \equiv (\phi, \phi)_\xi,\quad\xi = M,T,TM,MT.\\
			&\norm{\boldsymbol{u}}^2=\norm{u^x}^2_{TM}+\norm{u^y}^2_{MT},\\
			&\norm{\underline{\sigma}}^2=\norm{\sigma^{11}}_M^2+\norm{\sigma^{22}}_M^2+\norm{\sigma^{12}}_T^2.
		\end{align*}

%		\begin{align}
%			&\begin{aligned}
%				A\underline{\sigma}&=\frac{1}{\mu}(\underline{\sigma}-\frac{\lambda}{2\lambda+2\mu}(\sigma_{11}+\sigma_{22})\underline{I})\\
%				&=\begin{pmatrix}
%					\frac{\partial u_1}{\partial x_1} & \frac{1}{2}(\frac{\partial u_1}{\partial x_2}+\frac{\partial u_2}{\partial x_1})\\[1.5ex]
%					\frac{1}{2}(\frac{\partial u_1}{\partial x_2}+\frac{\partial u_2}{\partial x_1}) & \frac{\partial u_2}{\partial x_2}\\
%				\end{pmatrix}
%			\end{aligned}\\
%			&\begin{aligned}
%				div\underline{\sigma}=div
%				\begin{pmatrix}
%					\sigma_{11} & \sigma_{11}\\[1.5ex]
%					\sigma_{21} & \sigma_{21}\\
%				\end{pmatrix}=
%				\begin{pmatrix}
%					\frac{\partial \sigma_{11}}{\partial x_1} + \frac{\partial \sigma_{21}}{\partial x_2}\\[1.5ex]
%					\frac{\partial \sigma_{12}}{\partial x_1} + \frac{\partial \sigma_{22}}{\partial x_2}
%				\end{pmatrix}=
%				\begin{pmatrix}
%					f_1\\[1.5ex]
%					f_2
%				\end{pmatrix}
%			\end{aligned}
%		\end{align}

		\section{The MAC-E scheme and stability}In this section we present the MAC-E scheme for the elasticity problem and give the LBB condition and stability.

		Rewrite the model problem $\eqref{initial equations}$ as
		\begin{align}
			\frac{1}{2\mu}(\sigma^{11}-\frac{\lambda}{2\lambda+2\mu}(\sigma^{11}+\sigma^{22}))=\frac{\partial u^x}{\partial x},(x,y)\in \Omega,\label{sigma11equation}\\
			\frac{1}{\mu}\sigma^{12}=\frac{1}{\mu}\sigma^{21}=\frac{\partial u^x}{\partial y}+\frac{\partial u^y}{\partial x},(x,y)\in \Omega,\label{sigma12equation}\\
			\frac{1}{2\mu}(\sigma^{22}-\frac{\lambda}{2\lambda+2\mu}(\sigma^{11}+\sigma^{22}))=\frac{\partial u^y}{\partial y},(x,y)\in \Omega,\label{sigma22equation}\\
			\frac{\partial \sigma^{11}}{\partial x} + \frac{\partial \sigma^{21}}{\partial y}=f^x,(x,y)\in \Omega,\label{uxequation}\\
			\frac{\partial \sigma^{12}}{\partial x} + \frac{\partial \sigma^{22}}{\partial y}=f^y,(x,y)\in \Omega.\label{uyequation}
		\end{align}
		
		Denote by $\{W^x_{i,j+1/2}\},\{W^y_{i+1/2,j}\},\{Z^{11}_{i+1/2,j+1/2}\},\{Z^{12}_{i,j}\},$ and $ \{Z^{22}_{i+1/2,j+1/2}\} $ the approximations to $\{u^x_{i,j+1/2}\},\{u^y_{i+1/2,j}\},\{\sigma^{11}_{i+1/2,j+1/2}\},\{\sigma^{12}_{i,j}\}$ and $\{\sigma^{22}_{i+1/2,j+1/2}\}$,respectively.The scheme is as follows.It is worth noting that,since $\sigma^{12}=\sigma^{21}$,we use $\{Z^{12}_{i,j}\}$ to denote both $\{Z^{12}_{i,j}\}$ and $\{Z^{21}_{i,j}\}$.
		
		\textbf{MAC-E Scheme}:Set the boundary condition as
		\begin{equation}
			\begin{dcases}
				W^x_{0,j+1/2}=W^x_{n_x,j+1/2}=0,&j=0,\cdots,n_y-1,\\
				W^x_{i,0}=W^x_{i,n_y}=0,&i=0,\cdots,n_y,\\
				W^y_{i+1/2,0}=W^y_{i+1/2,n_y}=0,&i=0,\cdots,n_x-1,\\
				W^y_{0,j}=W^y_{n_x,j}=0,&j=0,\cdots,n_y.
			\end{dcases}
		\end{equation}
		Find  $\{W^x_{i,j+1/2}\},\{W^y_{i+1/2,j}\},\{Z^{11}_{i+1/2,j+1/2}\},\{Z^{12}_{i,j}\}$ and $\{Z^{22}_{i+1/2,j+1/2}\}$ such that
		\begin{align}
			&\begin{aligned}
				\frac{1}{2\mu}(Z^{11}_{i+1/2,j+1/2}-\frac{\lambda}{2\lambda+2\mu}(Z^{11}_{i+1/2,j+1/2}+Z^{22}_{i+1/2,j+1/2}))&-d_xW^x_{i+1/2,j+1/2}=0,\\
				i=0,\cdots,n_x-1,j&=0,\cdots,n_y-1,
			\end{aligned}\label{Z11equation}\\
			&\begin{aligned}
				Z^{12}_{i,j}-\mu(D_yW^x_{i,j}+D_xW^y_{i,j})=0,\qquad i=0,\cdots,n_x,j=0,\cdots,n_y,
			\end{aligned}\label{Z12equation}\\
			&\begin{aligned}
				\frac{1}{2\mu}(Z^{22}_{i+1/2,j+1/2}-\frac{\lambda}{2\lambda+2\mu}(Z^{11}_{i+1/2,j+1/2}+Z^{22}_{i+1/2,j+1/2}))&-d_yW^y_{i+1/2,j+1/2}=0,\\
				i=0,\cdots,n_x-1,j&=0,\cdots,n_y-1,
			\end{aligned}\label{Z22equation}\\
			&\begin{aligned}
				D_xZ^{11}_{i,j+1/2}+d_yZ^{12}_{i,j+1/2}=f^1_{i,j+1/2},\quad i=1,\cdots,n_x-1,j=0,\cdots,n_y-1,
			\end{aligned}\label{WXequation}\\
			&\begin{aligned}
				d_xZ^{12}_{i+1/2,j}+D_yZ^{22}_{i+1/2,j}=f^2_{i+1/2,j},\quad i=0,\cdots,n_x-1,j=1,\cdots,n_y-1.
			\end{aligned}\label{WYequation}
		\end{align}
		
		To obtain the stability and error estimate we need the discrete LBB condition.
		First let
		\begin{align}
			\underline{\Sigma}:=\{\underline{\tau}\in\underline{H}(div;\Omega) \},\boldsymbol{U}:=\boldsymbol{L}^2(\Omega).
		\end{align}
		
		\begin{figure}[htbp]
			\centering
			\begin{tabular}{@{}ccc@{}}
				\subcaptionbox{$\mathcal{T}_h^1$}[.24\linewidth]{%
					\includegraphics[width=\linewidth,trim=2pt 2pt 2pt 2pt,clip]{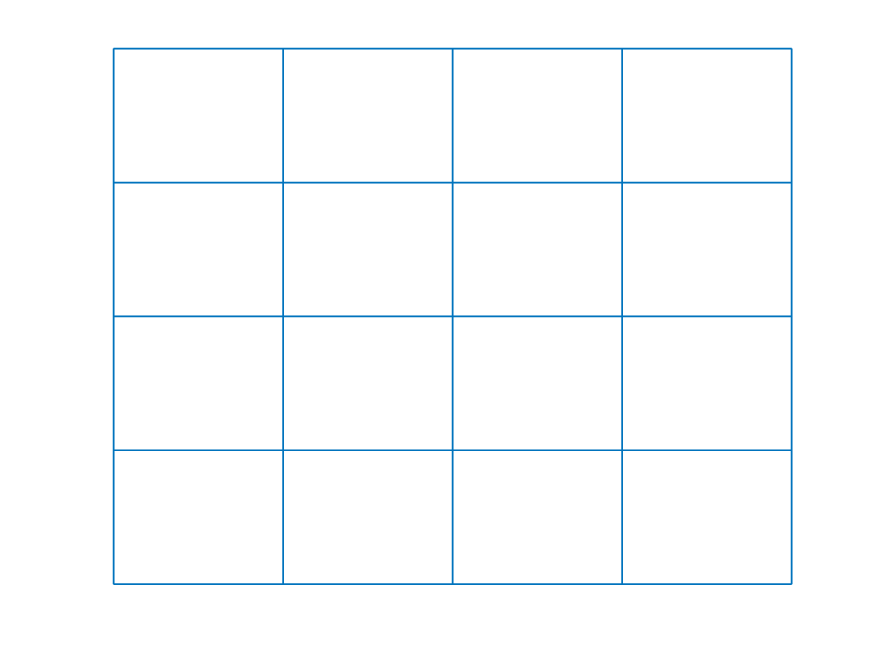}} &
				\subcaptionbox{$\mathcal{T}_h^2$}[.24\linewidth]{%
					\includegraphics[width=\linewidth,trim=2pt 2pt 2pt 2pt,clip]{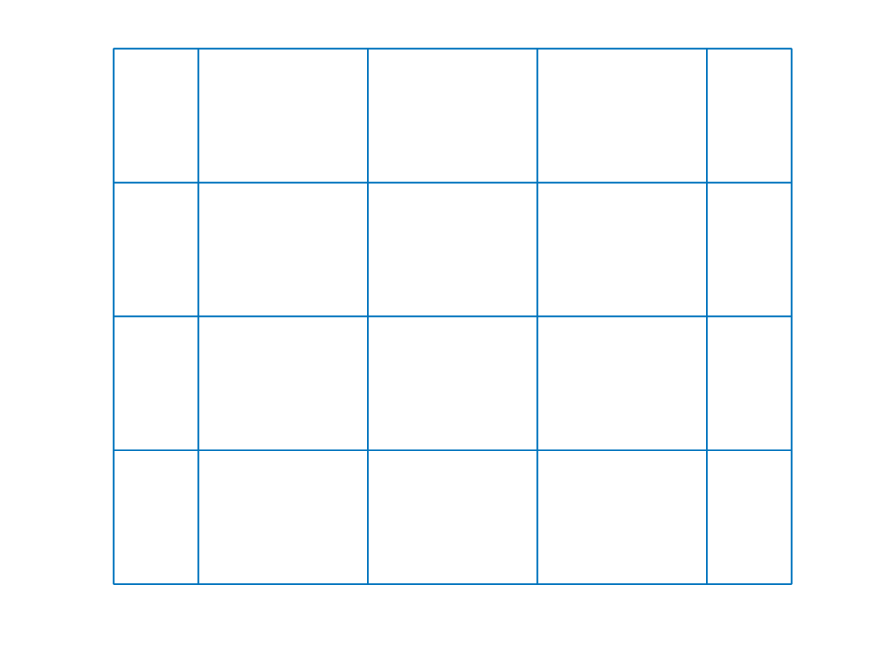}} &
				\subcaptionbox{$\mathcal{T}_h^3$}[.24\linewidth]{%
					\includegraphics[width=\linewidth,trim=2pt 2pt 2pt 2pt,clip]{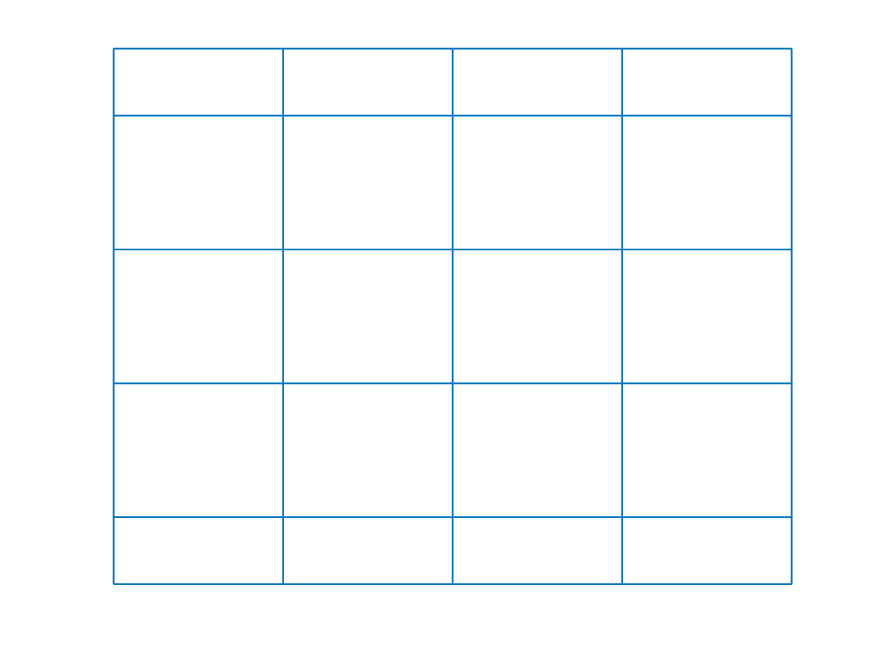}}
			\end{tabular}
			\caption{Three different partitions}
		\end{figure}
			
		Now we construct the finite-dimensional subspaces of $\underline{\Sigma}$ and $\boldsymbol{U}$ on three partitions
		$\mathcal{T}^1_h,\mathcal{T}^2_h,\mathcal{T}^3_h$ of $\Omega$ as follows.
		The original partition $\delta_x \times \delta_y$ is denoted by $\mathcal{T}^1_h$. The partition $\mathcal{T}_h^2$ is generated by connecting all the midpoints of the
		vertical sides of $\Omega_{i+1/2,\,j+1/2}$ and extending the resulting mesh to the
		boundary $\Gamma$. 
		To describe the local behavior of the discrete stress components, we further distinguish interior and boundary cells of $\mathcal{T}_h^2$:
		\begin{align*}
			&\mathcal{T}_{h,int}^2=\{\Omega_{i,j+1/2}=(x_{i-1/2},x_{i+1/2})\times(y_{j},y_{j+1})\in \mathcal{T}_h^2|i=1,\cdots,n_x-1,j=0,\cdots,n_y-1.\},\\
			&\mathcal{T}_{h,\partial_1}^2=\{\Omega_{0,j+1/2}=(x_0,x_{1/2})\times(y_{j},y_{j+1})\in \mathcal{T}_h^2|j=0,\cdots,n_y-1.\},\\
			&\mathcal{T}_{h,\partial_2}^2=\{\Omega_{n_x,j+1/2}=(x_{n_x-1/2},x_{n_x})\times(y_{j},y_{j+1})\in \mathcal{T}_h^2|j=0,\cdots,n_y-1.\}.
		\end{align*}
				
		Then, connecting all the midpoints of the horizontal sides
		of $\Omega_{i+1/2,\,j+1/2}$ and extending the resulting mesh to the boundary $\Gamma$,
		the third partition $\mathcal{T}_h^3$ is obtained.Similarly, we define:
		\begin{align*}
			&\mathcal{T}_{h,int}^3=\{\Omega_{i+1/2,j}=(x_{i},x_{i+1})\times(y_{j-1/2},y_{j+1/2})\in \mathcal{T}_h^3|i=0,\cdots,n_x-1,j=1,\cdots,n_y-1.\},\\
			&\mathcal{T}_{h,\partial_1}^3=\{\Omega_{i+1/2,0}=(x_i,x_{i+1})\times(y_0,y_{1/2})\in \mathcal{T}_h^3|i=0,\cdots,n_x-1.\},\\
			&\mathcal{T}_{h,\partial_2}^3=\{\Omega_{i+1/2,n_y}=(x_i,x_{i+1})\times(y_{n_y-1/2},y_{n_y})\in \mathcal{T}_h^3|i=0,\cdots,n_x-1.\}.
		\end{align*}
		
		Next, we introduce the following two function spaces.
		
		First, we use the notation $C^{m,n}(K)$ to denote the set of
		functions which are $m$-times continuously differentiable with respect to $x$ and $n$-times
		continuously differentiable with respect to $y$ on $K$. Moreover, we use the $C^{-1,n}(\Omega)$ and $C^{m,-1}(\Omega)$
		to denote spaces where no continuity is imposed in the $x$- and $y$-directions, respectively.
		
		We further introduce the tensor--product polynomial space $Q^{m,n}(\Omega)$.
		For integers $m,n\ge0$, we define
		\begin{equation*}
			Q^{m,n}(\Omega)
			:= \Bigl\{
			p : \Omega \to \mathbb{R} \,\Big|\,
			p(x,y) = \sum_{i=0}^{m}\sum_{j=0}^{n} a_{ij} x^{i} y^{j},
			\ a_{ij}\in\mathbb{R}
			\Bigr\}.
		\end{equation*}
		In other words, $Q^{m,n}(\Omega)$ consists of all polynomials which are of degree at
		most $m$ in $x$ and at most $n$ in $y$.

		We then define two global finite-dimensional spaces for these components as follows:
		\begin{align*}
			&\begin{aligned}
				S^1_h := 
				\{v\in C^{0,-1}(\Omega)|&\; 
				v|_K \in Q^{1,0}(K),\forall K \in \mathcal{T}^2_{h,int};v|_K=v(x_{1/2},y);\forall K\in \mathcal{T}^2_{h,\partial_1};\\&\;v|_K=v(x_{n_x-1/2},y);\forall K\in \mathcal{T}^2_{h,\partial_2}\},
			\end{aligned}\\
			&\begin{aligned}
				S^2_h := 
				\{v\in C^{-1,0}(\Omega)|&\; 
				v|_K \in Q^{0,1}(K),\forall K \in \mathcal{T}^3_{h,int};v|_K=v(x,y_{1/2});\forall K\in \mathcal{T}^3_{h,\partial_1};\\&\;v|_K=v(x,y_{n_y-1/2});\forall K\in \mathcal{T}^3_{h,\partial_2}\}.
			\end{aligned}
		\end{align*}
		With the above preparations,we construct a subspace of $\underline{\Sigma}$:
		\begin{align*}
			&\begin{aligned}
				\underline{\Sigma}_h:=\{\underline{\tau}_h\in\underline{H}(div;\Omega) :
				&\tau_h^{11}\in S^1_h(\Omega),\tau_h^{22}\in S^2_h(\Omega);\\
				&\tau_h^{12}|_T=\tau_h^{21}|_T\in Q^{1,1}(T),\forall T\in \mathcal{T}_h^1; \}.
			\end{aligned}
		\end{align*}
		
		Based on the quadrangulation $\mathcal{T}^2_h$ and $\mathcal{T}^3_h$, we construct a subspace of $\boldsymbol{U}$
		\begin{align*}
			\boldsymbol{U}_h=\{(v^x,v^y)\in \boldsymbol{L}^2(\Omega):&v^x|_T\equiv constant,\forall T\in \mathcal{T}_h^2,v_x|_\Gamma=0;\\&v^y|_T\equiv constant,\forall T\in \mathcal{T}_h^3,v_y|_\Gamma=0;\}.
		\end{align*}
		
		Then we introduce the bilinear forms:
		\begin{equation}
			\begin{aligned}
				\;b_h(\underline{\tau},\boldsymbol{W})=(W^x,D_x\tau^{11}+d_y\tau^{12})_{TM}+(W^y,D_y\tau^{22}+d_x\tau^{12})_{MT},
			\end{aligned}
		\end{equation}
		Next we prove that the LBB holds for the subspaces $\underline{\Sigma}_h$ and $\boldsymbol{U}_h$.
		
		\begin{lemma}
			There is a constant $\beta>0$ independent of $h$ and $l$ such that
			\begin{align}
				\sup_{0\neq\underline{\tau}_h\in\underline{\Sigma}_h}\frac{(D_x\tau_h^{11}+d_y\tau_h^{21},v_h^x)_{TM}+(d_x\tau_h^{12}+D_y\tau_h^{22},v_h^y)_{MT}}{\norm{\underline{\tau}_h}}\geq\beta \norm{\boldsymbol{v}_h}.		
			\end{align}
				\label{lemma3.1}
		\end{lemma}
			\begin{proof}
				Given $v=(v^x,v^y)^T\in \boldsymbol{U}_h$,we define $\tau_{11}$ as the integration of $v^x$ along the rectangles along the x direction:
				\begin{equation}
					\tau_h^{11}(x_{i+1/2},y_{j+1/2})=\int_{0}^{x_{i+1/2}}v_h^x(x,y_{j+1/2})dx,
				\end{equation}
				\begin{equation}
					\tau_h^{22}(x_{i+1/2},y_{j+1/2})=\int_{0}^{y_{j+1/2}}v_h^y(x_{i+1/2},y)dy.
				\end{equation}
				Then 
				\begin{equation}
					\begin{aligned}
						D_x\tau_h^{11}(x_{i},y_{j+1/2})&=\frac{1}{h_i}(\tau^{11}_h(x_{i+1/2,j+1/2})-\tau^{11}_h(x_{i-1/2,j+1/2}))\\
						&=\frac{1}{h_i}\int_{x_{i-1/2}}^{x_{i+1/2}}v_h^x(x,y_{j+1/2})dx=v^x_h(x_i,y_{j+1/2}),
					\end{aligned}
				\end{equation}
				Similarly we have 
				\begin{equation}
					D_y\tau_h^{22}(x_{i+1/2},y_j)=v^y_h(x_{i+1/2},y_j).
				\end{equation}
				Then, set $\tau_h^{12}=\tau_{h}^{21}=0$, it follows that
				\begin{align}
					D_x\tau_h^{11}(x_i,y_{j+1/2})+d_y\tau_h^{21}(x_i,y_{j+1/2})=D_x\tau_h^{11}(x_i,y_{j+1/2})=v_h^xD_x(x_i,y_{j+1/2}),\\
					d_x\tau_h^{12}(x_{i+1/2},y_j)+D_y\tau_h^{22}(x_{i+1/2},y_j)=D_y\tau_h^{22}(x_{i+1/2},y_j)=v_h^y(x_{i+1/2},y_j).
				\end{align}
				
				For the $L^2$-norm of $\tau_h^{11}$, we have:
				\begin{align}
					\begin{aligned}
						\norm{\tau_h^{11}}^2_{L^2(\Omega)}&=\sum_{i=1}^{n_x-1}\sum_{j=0}^{n_y-1}h_{i+1/2}y_{j+1/2}(\int_0^{x_{i+1/2}}v^x_h(x,y_{j+1/2})dx)^2\\
						&=\sum_{i=1}^{n_x-1}\sum_{j=0}^{n_y-1}h_{i+1/2}y_{j+1/2}(\sum_{k=1}^iv^x_h(x_k,y_{j+1/2})h_{k})^2\\
						&\leq\sum_{i=1}^{n_x-1}\sum_{j=0}^{n_y-1}h_{i+1/2}y_{j+1/2}\sum_{k=1}^ih_k(v^x_h(x_k,y_{j+1/2}))^2\\
						&=\sum_{k=1}^{n_x-1}\sum_{j=0}^{n_y-1}h_{k}y_{j+1/2}(v^x_h(x_k,y_{j+1/2}))^2\sum_{i=k}^{n_x-1}h_{i+1/2}\\
						&\leq\sum_{k=1}^{n_x-1}\sum_{j=0}^{n_y-1}h_{k}y_{j+1/2}(v^x_h(x_k,y_{j+1/2}))^2=\norm{v^x_h}^2_{L^2(\Omega)}.
					\end{aligned}\label{equation3.19}
				\end{align}
				
				For y-direction,a similar argument with $\eqref{equation3.19}$ proves that
				\begin{equation*}
					\norm{\tau_h^{22}}^2_{L^2(\Omega)}\leq\norm{v^y_h}^2_{L^2(\Omega)}.
				\end{equation*}
			Then we could get that:
			\begin{equation}
				\frac{(D_x\tau_h^{11}+d_y\tau_h^{21},v_h^x)_{TM}+(d_x\tau_h^{12}+D_y\tau_h^{22},v_h^y)_{MT}}{\norm{\underline{\tau}_h}^2}=\frac{\norm{\boldsymbol{v}_h}^2}{\norm{\tau_h^{11}}^2_M+\norm{\tau_h^{22}}^2_M}\geq\norm{\boldsymbol{v}_h}^2.
				\label{equation3.20}
			\end{equation}
			Finally, taking the supremum over $\underline{\tau
			}_h$ on the left-hand side of the equation $\eqref{equation3.20}$, we obtain the Lemma $\ref{lemma3.1}$
			\end{proof}
			
	\begin{remark}
		It should be noted that for a uniform mesh, $\eqref{equation3.19}$ becomes
		\begin{equation*}
			\norm{\tau_h^{11}}^2_{L^2(\Omega)}\leq\frac{1}{2}\norm{v^x_h}^2_{L^2(\Omega)}.
		\end{equation*}
		And inf-sup condition will become:
		\begin{align}
			\sup_{0\neq\underline{\tau}_h\in\underline{\Sigma}_h}\frac{(D_x\tau_h^{11}+d_y\tau_h^{21},v_h^x)_{TM}+(d_x\tau_h^{12}+D_y\tau_h^{22},v_h^y)_{MT}}{\norm{\tau_h}}\geq\frac{\sqrt{2}}{2} \norm{\boldsymbol{v}_h}.		
		\end{align}
	\end{remark}
		
		\begin{lemma}
			Let $\{\sigma^{11}_{i+1/2,j+1/2}\},\{\tau^{22}_{i+1/2,j+1/2}\},\{\tau^{12}_{i,j}\}\;and\;\{v^{x}_{i,j+1/2}\},\{v^{y}_{i+1/2,j}\}$ be discrete functions with $v^x_{0,j+1/2}=v^x_{0,j}=v^x_{n_x,j+1/2}=v^x_{n_x,j}=v^y_{i+1/2,0}=v^y_{i,0}=v^y_{i+1/2,n_y}=v^y_{i,n_y}=0$, with proper integral i,j. Then there holds
			\begin{equation}
				\begin{dcases}
					&(D_x\tau^{11},v^x)_{TM}=-(\tau^{11},d_xv^x)_M,\\
					&(D_y\tau^{22},v^y)_{MT}=-(\tau^{22},d_yv^y)_M,\\
					&(d_xv^x,\tau^{11})_M=-(v^x,D_x\tau^{11})_{TM},\\
					&(d_yv^y,\tau^{22})_M=-(v^y,D_y\tau^{22})_{MT},\\
					&(D_xv^y+D_yv^x,\tau^{12})_T=-(v^y,d_x\tau^{12})_{MT}-(v^x,d_y\tau^{12})_{TM}.
				\end{dcases}
			\end{equation}
			\label{lemma3.2}
		\end{lemma}
		
		\begin{proof}
			The derivations of the first four items follow directly from the same approach used in case \cite{weiser1988convergence}. As for the fifth item, we have:
			\begin{equation}
				\begin{aligned}
					(D_xv^y,\tau^{12})_T=&\sum_{i=0}^{n_x}\sum_{j=0}^{n_y}D_xv^y_{i,j}\tau^{12}_{i,j}h_il_j\\
					=&\sum_{j=1}^{n_y-1}\sum_{i=1}^{n_x-1}D_xv^y_{i,j}\tau^{12}_{i,j}h_il_j+\sum_{j=1}^{n_y-1}(D_xv^y_{0,j}\tau^{12}_{0,j}h_0l_j+D_xv^y_{n_x,j}\tau^{12}_{n_x,j}h_{n_x}l_j)\\
					=&-\sum_{j=1}^{n_y-1}\sum_{i=0}^{n_x-1} v^y_{i+1/2,j}d_x\tau^{12}_{i+1/2,j}h_{i+1/2}l_j-\sum_{j=0}^{n_y}v^y_{1/2,j}\tau^{12}_{0,j}l_j+\sum_{j=0}^{n_y}v^y_{n_x-1/2,j}\tau^{12}_{n_x,j}l_j\\
					&+\sum_{j=1}^{n_y-1}(v^y_{1/2,j}-v^y_{0,j})\tau^{12}_{0,j}l_j+(v^y_{n_x,j}-v^y_{n_x-1/2,j})\tau^{12}_{n_x,j}l_j\\
					=&-\sum_{i=0}^{n_x-1}\sum_{j=1}^{n_y-1} v^y_{i+1/2,j}d_x\tau^{12}_{i+1/2,j}h_{i+1/2}l_j=-(v^y,d_x\tau^{12})_{MT}.
				\end{aligned}
			\end{equation}
			The proof is completed.
		\end{proof}
		
		Then we set:
		\begin{align}
			&\begin{aligned}
				a_h(\lambda;\underline{Z},\underline{\tau})=&\frac{1}{\mu}(Z^{12},\tau^{12})_T-\frac{1}{2\mu}\frac{\lambda}{2\lambda+2\mu}((Z^{11},\tau^{22})_M+(Z^{22},\tau^{11})_M)\\
				+&\frac{1}{2\mu}(1-\frac{\lambda}{2\lambda+2\mu})((Z^{11},\tau^{11})_M+(Z^{22},\tau^{22})_M),\\
			\end{aligned}\label{ah}\\
			&\begin{aligned}
				\quad\;\;\;(\boldsymbol{f},\boldsymbol{v})=(f^1,v^x)_{TM}+(f^2,v^y)_{MT}.
			\end{aligned}
		\end{align}
		
		\begin{lemma}
			For $\underline{Z},\underline{\tau}\in \underline{\Sigma}_h$,there exist $C_1(\mu)$ so that
			\begin{align}
				&a_h(\lambda;\underline{Z},\underline{Z})\geq \frac{1}{\mu}\norm{Z^{12}}_T^2+\frac{1}{2\lambda+2\mu}(\norm{Z^{11}}_M^2+\norm{Z^{22}}_M^2),\label{equation3.24}\\
				&a_h(\lambda;\underline{Z},\underline{\tau})\leq C_1(\mu)\norm{\underline{Z}}\norm{\underline{\tau}}.\label{equation3.25}
			\end{align}
			\label{lemma3.3}
		\end{lemma}
		\begin{proof}
			Let $\underline{\tau}=\underline{Z}$ in $a_h$,we can obtain
			\begin{align}
				\begin{aligned}
					a_h(\lambda;\underline{Z},\underline{Z})&=\frac{1}{\mu}\norm{Z^{12}}_T^2-\frac{1}{2\mu}\frac{\lambda}{\lambda+\mu}(Z^{11},Z^{22})_M+\frac{1}{2\mu}(1-\frac{\lambda}{2\lambda+2\mu})(\norm{Z^{11}}_M^2+\norm{Z^{22}}_M^2)\\
					&\geq\frac{1}{\mu}\norm{Z^{12}}_T^2+\frac{1}{2\lambda+2\mu}(\norm{Z^{11}}_M^2+\norm{Z^{22}}_M^2).
				\end{aligned}
			\end{align}
			So we get the $\eqref{equation3.24}$.
			On the other hand:
			\begin{align}
					a_h(\lambda;\underline{Z},\underline{\tau})=&\frac{1}{\mu}(Z^{12},\tau^{12})_T-\frac{1}{2\mu}\frac{\lambda}{2\lambda+2\mu}((Z^{11},\tau^{22})_M+(Z^{22},\tau^{11})_M)\notag\\
					&+\frac{1}{2\mu}(1-\frac{\lambda}{2\lambda+2\mu})((Z^{11},\tau^{11})_M+(Z^{22},\tau^{22})_M)\notag\\
					\leq&\frac{1}{\mu}\norm{Z^{12}}_T\norm{\tau^{12}}_T+\frac{1}{2\mu}\frac{\lambda}{2\lambda+2\mu}(\norm{Z^{11}}_M\norm{\tau^{22}}_M+\norm{Z^{22}}_M\norm{\tau^{11}}_M)\notag\\
					&+\frac{1}{2\mu}(1-\frac{\lambda}{2\lambda+2\mu})(\norm{Z^{11}}_M\norm{\tau^{11}}_M+\norm{Z^{22}}_M\norm{\tau^{22}}_M)\notag\\
					\leq&\frac{1}{2\mu}(2\norm{Z^{12}}_T\norm{\tau^{12}}_T+\norm{Z^{11}}_M\norm{\tau^{11}}_M+\norm{Z^{22}}_M\norm{\tau^{22}}_M\notag\\
					&\quad+\norm{Z^{11}}_M\norm{\tau^{22}}_M+\norm{Z^{22}}_M\norm{\tau^{11}}_M)\notag\\
					\leq&\frac{1}{2\mu}(2\norm{Z^{12}}_T^2+\norm{Z^{11}}_M^2+\norm{Z^{22}}_M^2)^{1/2}(2\norm{\tau^{12}}_T^2+\norm{\tau^{11}}_M^2+\norm{\tau^{22}}_M^2)^{1/2}\notag\\
					=&\frac{1}{2\mu}\norm{\underline{Z}}\norm{\underline{\tau}}=C_1(\mu)\norm{\underline{Z}}\norm{\underline{\tau}}.
				\label{equation3.29}
			\end{align}
			The proof is completed.
		\end{proof}

		From the above derivations, we can obtain the following conclusion:
		\begin{corollary}
			For $a_h(\underline{Z},\underline{Z})$ and $ \frac{1}{\mu}\norm{Z^{12}}_T^2+\frac{1}{2\lambda+2\mu}(\norm{Z^{11}}_M^2+\norm{Z^{22}}_M^2)$,there exist $C_2(\mu
			)$ and $C_3(\mu)$,independent of h and k such that
			\begin{equation}
				C_2(\mu)a_h(\lambda;\underline{Z},\underline{Z})\leq\frac{1}{\mu}\norm{Z^{12}}_T^2+\frac{1}{2\lambda+2\mu}(\norm{Z^{11}}_M^2+\norm{Z^{22}}_M^2)\leq  C_3(\mu)a_h(\lambda;\underline{Z},\underline{Z}).
			\end{equation}
		\end{corollary}
		
		Combining $Lemma\,\ref{lemma3.3}$, and a similar detailed derivation will be presented later in the section 4, then the scheme $\eqref{Z11equation}-\eqref{WYequation}$ can be written as 
		\begin{align}
			a_h(\lambda;\underline{Z},\underline{\tau})+b_h(\underline{\tau},\boldsymbol{W})&=0,\forall \tau\in \underline{\Sigma}_h,\\
			b_h(\underline{Z},\boldsymbol{v})&=(\boldsymbol{f},\boldsymbol{v}),\forall \boldsymbol{v}\in\boldsymbol{U}_h.
		\end{align}

		From Lemmas $\eqref{lemma3.1},\eqref{lemma3.3}$ and the abstract theory for saddle point problem Girault and Raviart \cite{girault2012finite} we can obtain the next theorem.
		
		\begin{theorem}
			For the solution of $\eqref{Z11equation}-\eqref{WYequation}$ there exists a positive constant $C_4(\mu)$ independent
			of h and l such that
			\begin{equation}
				\frac{1}{\mu}\norm{Z^{12}}_T+\frac{1}{2\lambda+2\mu}(\norm{Z^{11}}_M+\norm{Z^{22}}_M)+\norm{\boldsymbol{W}}\leq C_4(\mu)(\norm{f^1}_{TM}+\norm{f^2}_{MT}).
			\end{equation}
		\end{theorem}
		
		\section{Superconvergence analysis.}In this section with the analytical solution $\underline{\sigma}$ and $\boldsymbol{u}$ sufficiently smooth, we obtain some superconvergence results about the stress and displacement. We give some lemmas first. The partition parameters $h_i$ and $l_j$ can be looked at as discrete functions.
		
		Define the discrete interpolations of $\sigma^{11},\sigma^{22}$ and $\sigma^{12}$ in $\Omega$ and on the boundary $\partial \Omega$ as follows:
		\begin{align}
			&\tilde{\sigma}^{11}_{i+1/2,j+1/2}=(\sigma^{11}-\delta^{11})_{i+1/2,j+1/2},\quad 0\leq i\leq n_x-1,1\leq j\leq n_y-1,\label{equation4.1}\\
			&\tilde{\sigma}^{22}_{i+1/2,j+1/2}=(\sigma^{22}-\delta^{22})_{i+1/2,j+1/2},\quad 0\leq i\leq n_x-1,1\leq j\leq n_y-1,\label{equation4.2}\\
			&\tilde{\sigma}^{12}_{i,j}=\sigma^{12}_{i,j},\quad \qquad\qquad\qquad\qquad\qquad\quad\;\; 0\leq i\leq n_x,1\leq j\leq n_y.\label{equation4.3}
		\end{align}
		where
		\begin{align}
			&\delta^{11}_{i+1/2,j+1/2}=\frac{h^2_{i+1/2}}{8}\frac{\partial^2\sigma^{11}_{i+1/2,j+1/2}}{\partial x^2}+\frac{l^2_{j+1/2}}{8}\frac{\partial^2\sigma^{11}_{i+1/2,j+1/2}}{\partial y^2},\label{equation4.4}\\
			&\delta^{22}_{i+1/2,j+1/2}=\frac{h^2_{i+1/2}}{8}\frac{\partial^2\sigma^{22}_{i+1/2,j+1/2}}{\partial x^2}+\frac{l^2_{j+1/2}}{8}\frac{\partial^2\sigma^{22}_{i+1/2,j+1/2}}{\partial y^2}.\label{equation4.5}
		\end{align}
		 Then, using Taylor’s expansion similar to \cite{rui2017stability}, we obtain the following lemma.
		\begin{lemma}
			If $\underline{\sigma}$ is sufficiently smooth then for $\sigma^{11},\sigma^{22}$ there hold
			\begin{equation}
				\begin{dcases}
					(\frac{\partial \sigma^{11}}{\partial x})_{i,j+1/2}=D_x\tilde{\sigma}^{11}_{i,j+1/2}+\epsilon^x_{i,j+1/2}(\sigma^{11}),\\
					(\frac{\partial \sigma^{22}}{\partial y})_{i+1/2,j}=D_y\tilde{\sigma}^{22}_{i+1/2,j}+\epsilon^y_{i+1/2,j}(\sigma^{22}).
				\end{dcases}
			\end{equation}
			with the approximate properties  
			\begin{equation}
				\epsilon^x_{i,j+1/2}(\sigma^{11})=O((h^2+l^2)\norm{\sigma^{11}}_{3,\infty}),\epsilon^y_{i+1/2,j}(\sigma^{22})=O((h^2+l^2)\norm{\sigma^{22}}_{3,\infty}).
			\end{equation}
			\label{lemma 4.1}
		\end{lemma}
		
		Define the discrete interpolation of $u^x$ and $u^y$ in $\Omega$ and on the boundary $\partial \Omega$ as follows.
		\begin{align}
			&\begin{aligned}
				&\tilde{u}^x_{i,j+1/2}=u^x_{i,j+1/2}-\frac{l^2_{j+1/2}}{8}\frac{\partial^2u^x_{i,j+1/2}}{\partial y^2},&&0\leq i\leq n_x,0\leq j\leq n_y-1,\\
				&\tilde{u}^x_{i,j}=u^x_{i,j}-\frac{l^2_{j}}{8}\frac{\partial^2u^x_{i,j}}{\partial y^2},&&0\leq i\leq n_x,j=0,n_y,
				\label{equation4.8}
			\end{aligned}\\		
			&\begin{aligned}
				&\tilde{u}^y_{i+1/2,j}=u^y_{i+1/2,j}-\frac{h^2_{i+1/2}}{8}\frac{\partial^2u^y_{i+1/2,j}}{\partial x^2},&&0\leq i\leq n_x-1,0\leq j\leq n_y,\\
				&\tilde{u}^y_{i,j}=u^y_{i,j}-\frac{h^2_{i}}{8}\frac{\partial^2u^y_{i,j}}{\partial x^2},&&i=0,n_x,0\leq j\leq n_y.
				\label{equation4.9}
			\end{aligned}
		\end{align}

		Similarly to \cite{rui2017stability}, we have
		\begin{lemma}
			If $\boldsymbol{u}$ is sufficiently smooth then there hold
			\begin{equation}
				\begin{dcases}
					(\frac{\partial u^x}{\partial y})_{i,j}=D_y\tilde{u}^x_{i,j}+\epsilon^x_{i,j}(u^x),\quad i=0,\cdots,n_x,j=0,\cdots,n_y,\\
					(\frac{\partial u^y}{\partial x})_{i,j}=D_x\tilde{u}^y_{i,j}+\epsilon^y_{i,j}(u^y),\quad i=0,\cdots,n_x,j=0,\cdots,n_y,\\
					\begin{aligned}
						(\frac{\partial u^x}{\partial x})_{i+1/2,j+1/2}=d_x\tilde{u}^x_{i+1/2,j+1/2}&+\epsilon^x_{i+1/2,j+1/2}(u^x),\\
						i=&0,\cdots,n_x-1,j=0,\cdots,n_y-1,
					\end{aligned}\\
					\begin{aligned}
						(\frac{\partial u^y}{\partial x})_{i+1/2,j+1/2}=d_y\tilde{u}^y_{i+1/2,j+1/2}&+\epsilon^y_{i+1/2,j+1/2}(u^y),\\	
						i=&0,\cdots,n_x-1,j=0,\cdots,n_y-1.
					\end{aligned}
				\end{dcases}
			\end{equation}
			with the approximate properties  
			\begin{equation}
				\begin{dcases}
					\epsilon^x_{i,j}(u^x)=O((h^2+l^2)\norm{u^x}_{3,\infty}),&\epsilon^y_{i,j}(u^y)=O((h^2+l^2)\norm{u^y}_{3,\infty}),\\
					\epsilon^x_{i+1/2,j+1/2}(u^x)=O((h^2+l^2)\norm{u^x}_{3,\infty}),&\epsilon^y_{i+1/2,j+1/2}(u^y)=O((h^2+l^2)\norm{u^y}_{3,\infty}).
				\end{dcases}
			\end{equation}
			\label{lemma4.2}
		\end{lemma}

		Now we consider the superconvergence analysis. Set
		\begin{equation}
			\begin{dcases}
				\begin{aligned}
					&E^{11}_{i+1/2,j+1/2}=(Z^{11}-\tilde{\sigma}^{11})_{i+1/2,j+1/2},&&i=0,1,\cdots,n_x-1,j=0,1,\cdots,n_y-1,\\
					&E^{22}_{i+1/2,j+1/2}=(Z^{22}-\tilde{\sigma}^{22})_{i+1/2,j+1/2},&&i=0,1,\cdots,n_x-1,j=0,1,\cdots,n_y-1,\\
					&E^{12}_{i,j}=(Z^{12}-\tilde{\sigma}^{12})_{i,j},&&i=0,\cdots,n_x,j=0,\cdots,n_y,\\
					&E^x_{i,j+1/2}=(W^x-\tilde{u}^x)_{i,j+1/2},&&i=1,\cdots,n_x-1,j=0,\cdots,n_y-1,\\
					&E^y_{i+1/2,j}=(W^y-\tilde{u}^y)_{i+1/2,j},&&i=0,\cdots,n_x-1,j=1,\cdots,n_y-1.\\
				\end{aligned}
			\end{dcases}
		\end{equation}
		
		Subtracting $\eqref{WXequation}$ from $\eqref{uxequation}$, we have
		\begin{equation}
			D_xE^{11}_{i,j+1/2}+d_yE^{12}_{i,j+1/2}=(\frac{\partial \sigma^{11}}{\partial x})_{i,j+1/2}-D_xZ^{11}_{i,j+1/2}+(\frac{\partial \sigma^{12}}{\partial y})_{i,j+1/2}-d_yZ^{12}_{i,j+1/2},
		\end{equation}
		From Lemma $\ref{lemma 4.1}$ and Taylor expansion,we could get
		\begin{equation}
			D_xE^{11}_{i,j+1/2}+d_yE^{12}_{i,j+1/2}=\epsilon^x_{i,j+1/2}(\sigma^{11})+\epsilon^{2y}_{i,j+1/2}(\sigma^{12}),\label{WXwuchaequation}
		\end{equation}
		Similarly we have
		\begin{equation}
			D_yE^{22}_{i+1/2,j}+d_x\tilde{E}^{12}_{i+1/2,j}=\epsilon^y_{i+1/2,j}(\sigma^{22})+\epsilon^{2x}_{i+1/2,j}(\sigma^{12}),\label{WYwuchaequation}
		\end{equation}
		where
		\begin{equation*}
			\begin{dcases}
				\epsilon^{2x}_{i+1/2,j}(\sigma^{12})=O((h^2+l^2)\norm{\sigma^{12}}_{3,\infty}),\\
				\epsilon^{2y}_{i,j+1/2}(\sigma^{12})=O((h^2+l^2)\norm{\sigma^{12}}_{3,\infty}).
			\end{dcases}
		\end{equation*}
		We set
		\begin{equation}
			\begin{dcases}
				r^x_{i,j+1/2}=\epsilon^x_{i,j+1/2}(\sigma^{11})+\epsilon^{2y}_{i,j+1/2}(\sigma^{12})=O((h^2+l^2)(\norm{\sigma^{11}}_{3,\infty}+\norm{\sigma^{12}}_{3,\infty})),\\
				r^y_{i+1/2,j}=\epsilon^y_{i+1/2,j}(\sigma^{22})+\epsilon^{2x}_{i+1/2,j}(\sigma^{12})=O((h^2+l^2)(\norm{\sigma^{22}}_{3,\infty}+\norm{\sigma^{12}}_{3,\infty})).
			\end{dcases}
		\end{equation}
		
		For the discrete function $\{v^x_{i,j+1/2}\}$ with $v^x_{i,j+1/2}|_{\partial \Omega}=0$,  $\eqref{WXwuchaequation}$ times $v^x_{i,j+1/2}h_il_{j+1/2}$ and summing $i,j$ for $1\leq i\leq n_x-1,0\leq j\leq n_y-1$, we can get
		\begin{equation}
			(D_xE^{11},v^x)_{TM}+(d_yE^{12},v^x)_{TM}=(r^x,v^x)_{TM}.\label{equation4.18}
		\end{equation}
			
		Similarly,	For the discrete function $\{v^y_{i+1/2,j}\}$ with $v^y_{i+1/2,j}|_{\partial \Omega}=0$,we have
		\begin{equation}
			(D_yE^{22},v^y)_{MT}+(d_xE^{12},v^y)_{MT}=(r^y,v^y)_{MT}.\label{equation4.19}
		\end{equation}
		Additionally,combining Lemma $\ref{lemma3.2}$ we have 
		\begin{align}
			(D_xE^{11},v^x)_{TM}+(d_yE^{12},v^x)_{TM}=-(E^{11},d_xv^x)_M-(E^{12},D_yv^x)_T,\label{equation4.20}\\
			(D_yE^{22},v^y)_{MT}+(d_xE^{12},v^y)_{MT}=-(E^{22},d_yv^x)_M-(E^{12},D_xv^x)_T.\label{equation4.21}
		\end{align}
		
		Subtracting $\eqref{sigma12equation}$ from $\eqref{Z12equation}$, and combining the Lemma $\ref{lemma 4.1}$ we have
		\begin{equation}
		 	\frac{1}{\mu}E^{12}_{i,j}-(D_yE^x_{i,j}+D_xE^y_{i,j})=\epsilon^x_{i,j}(u^x)+\epsilon^y_{i,j}(u^y).\label{equation4.22}
	    \end{equation}
	    We set
	    \begin{equation}
	    	r^{12}_{i,j}=\epsilon^x_{i,j}(u^x)+\epsilon^y_{i,j}(u^y)=O((h^2+l^2)(\norm{u^x}_{3,\infty}+\norm{u^y}_{3,\infty})).
	    \end{equation}
	    
	    Multiplying $\eqref{equation4.22}$ by $\tau^{12}_{i,j}h_il_j$ with $\tau^{12}_{i,j}|_{\partial \Omega}=0$,making summation for $i=1,2,\cdots,n_x-1,j=1,2,\cdots,n_y-1$, and combining the Lemma $\ref{lemma3.2}$ we have
	    \begin{equation}
	    	\begin{aligned}
	    		&\frac{1}{\mu}(E^{12},\tau^{12})_T-(D_yE^x,\tau^{12})_T-(D_xE^y,\tau^{12})_T\\
	    		=&\frac{1}{\mu}(E^{12},\tau^{12})_T+(E^x,d_y\tau^{12})_{TM}+(E^y,d_x\tau^{12})_{MT}=(r^{12},\tau^{12})_T.
	    	\end{aligned}\label{equation4.24}
	    \end{equation}
	    
	    Subtracting $\eqref{sigma11equation}$ from $\eqref{Z11equation}$, combining Lemma $\ref{lemma 4.1}$ and Lemma $\ref{lemma4.2}$ we have 
	    \begin{equation}
	    	\begin{aligned}
	    		&\frac{1}{2\mu}((1-\frac{\lambda}{2\lambda+2\mu})E^{11}_{i+1/2,j+1/2}-\frac{\lambda}{2\lambda+2\mu}E^{22}_{i+1/2,j+1/2}))-d_xE^x_{i+1/2,j+1/2}\\
	    		=&r^{11}_{i+1/2,j+1/2}=O((h^2+l^2)\norm{u^x}_{3,\infty}).
	    	\end{aligned}
	    	\label{E11wcequation}
	    \end{equation}
	    Multiplying $\eqref{E11wcequation}$ with $\tau^{11}_{i+1/2,j+1/2}h_{i+1/2}l_{j+1/2}$, and making summation, we can obtain
	    \begin{equation}
	    	\begin{aligned}
	    		&\frac{1}{2\mu}((1-\frac{\lambda}{2\lambda+2\mu})(E^{11},\tau^{11})_M-\frac{\lambda}{2\lambda+2\mu}(E^{22},\tau^{11})_M)-(d_xE^x,\tau^{11}	)_M\\
	    		=&\frac{1}{2\mu}(\frac{\lambda+2\mu}{2\lambda+2\mu}(E^{11},\tau^{11})_M-\frac{\lambda}{2\lambda+2\mu}(E^{22},\tau^{11})_M)+(E^x,D_x\tau^{11}	)_M\\
	    		=&(r^{11},E^{11})_M.
	    	\end{aligned}\label{equation4.26}
	    \end{equation}
	    
	    Subtracting $\eqref{sigma22equation}$ from $\eqref{Z22equation}$, combining $Lemma\,\ref{lemma 4.1}$ and $Lemma\,\ref{lemma4.2}$, we have
		\begin{equation}
			\begin{aligned}
				&\frac{1}{2\mu}((1-\frac{\lambda}{2\lambda+2\mu})E^{22}_{i+1/2,j+1/2}-\frac{\lambda}{2\lambda+2\mu}E^{11}_{i+1/2,j+1/2}))-d_yE^y_{i+1/2,j+1/2}\\
				=&r^{22}_{i+1/2,j+1/2}=O((h^2+l^2)\norm{u^y}_{3,\infty}).
			\end{aligned}
			\label{E22wcequation}
		\end{equation}Multiplying $\eqref{E22wcequation}$ with $\tau^{22}_{i+1/2,j+1/2}h_{i+1/2}l_{j+1/2}$, and making summation, we can obtain
	    \begin{equation}
	    	\begin{aligned}
	    		&\frac{1}{2\mu}((1-\frac{\lambda}{2\lambda+2\mu})(E^{22},\tau^{22})_M-\frac{\lambda}{2\lambda+2\mu}(E^{11},\tau^{22})_M))-(d_yE^y,\tau^{11}	)_M\\
	    		=&\frac{1}{2\mu}(\frac{\lambda+2\mu}{2\lambda+2\mu}(E^{22},\tau^{22})_M-\frac{\lambda}{2\lambda+2\mu}(E^{11},\tau^{22})_M)+(E^y,D_y\tau^{22}	)_M\\
	    		=&(r^{22},E^{22})_M.
	    	\end{aligned}\label{equation4.28}
	    \end{equation}

		Adding $\eqref{equation4.24},\eqref{equation4.26}$ and $\eqref{equation4.28}$ and combining $\eqref{equation3.29}$ results in
		\begin{equation}
			\begin{aligned}
				&-(E^x,D_x\tau^{11}	)_M-(E^y,D_y\tau^{22}	)_M-(E^x,d_y\tau^{12})_{TM}-(E^y,d_x\tau^{12})_{MT}\\
				=&\frac{1}{\mu}(E^{12},\tau^{12})_T-\frac{1}{2\mu}\frac{\lambda}{2\lambda+2\mu}((E^{11},\tau^{22})_M+(E^{22},\tau^{11})_M)\\
				&+\frac{1}{2\mu}\frac{\lambda+2\mu}{2\lambda+2\mu}((E^{11},\tau^{11})_M+(E^{22},\tau^{22})_M)-(r^{12},\tau^{12})_T-(r^{11},\tau^{11})_M-(r^{22},\tau^{22})_M\\
				\leq&\frac{1}{2\mu}\norm{\underline{E}}\norm{\underline{\tau}}+C\norm{\underline{r}}\norm{\underline{\tau}}.
			\label{equation4.29}
		\end{aligned}\\
		\end{equation}
		
		Using $\eqref{equation4.29}$ and Lemma $\ref{lemma3.1}$ we obtain that
		\begin{equation}
			\begin{aligned}
				\norm{\boldsymbol{E}}&\leq
				\sup_{0\neq\underline{\tau}\in\underline{\Sigma}_h}-\frac{(E^x,D_x\tau^{11}	)_M+(E^y,D_y\tau^{22}	)_M+(E^x,d_y\tau^{12})_{TM}+(E^y,d_x\tau^{12})_{MT}}{\norm{\underline{\tau}}}\\
				&\leq\frac{1}{2\mu}\norm{\underline{E}}+C\norm{\underline{r}}.
			\end{aligned}
			\label{equation4.30}
		\end{equation}
		
		Setting $\tau^{11}_{i+1/2,j+1/2}=E^{11}_{i+1/2,j+1/2},\tau^{22}_{i+1/2,j+1/2}=E^{22}_{i+1/2,j+1/2},\tau^{12}_{i,j}=E^{12}_{i,j}$ in $\eqref{equation4.29}$, we obtain that
		\begin{equation}
			\begin{aligned}
				&\frac{1}{\mu}\norm{E^{12}}_T^2-\frac{1}{\mu}\frac{\lambda}{2\lambda+2\mu}(E^{11},E^{22})_M+\frac{1}{2\mu}\frac{\lambda+2\mu}{2\lambda+2\mu}(\norm{E^{11}}_M^2+\norm{E^{22}}_M^2)\\
				=&-(E^x,D_xE^{11})_M-(E^y,D_yE^{22}	)_M-(E^x,d_yE^{12})_{TM}-(E^y,d_xE^{12})_{MT}\\
				&+(r^{12},E^{12})_T+(r^{11},E^{11})_M+(r^{22},E^{22})_M.
			\end{aligned}
			\label{equation4.31}
		\end{equation}
		On the left-hand side of the $\eqref{equation4.31}$, we have
		\begin{equation}
			\begin{aligned}
				&\frac{1}{\mu}\norm{E^{12}}_T^2-\frac{1}{\mu}\frac{\lambda}{2\lambda+2\mu}(E^{11},E^{22})_M+\frac{1}{2\mu}(1-\frac{\lambda}{2\lambda+2\mu})(\norm{E^{11}}_M^2+\norm{E^{22}}_M^2)\\
				\geq&\frac{1}{\mu}\norm{E^{12}}_T^2+\frac{1}{2\mu}(1-\frac{\lambda}{2\lambda+2\mu})(\norm{E^{11}}_M^2+\norm{E^{22}}_M^2)-\frac{1}{2\mu}\frac{\lambda}{2\lambda+2\mu}(\norm{E^{11}}_M^2+\norm{E^{22}}_M^2)\\
				=&\frac{1}{\mu}\norm{E^{12}}_T^2+\frac{1}{2\mu}(1-\frac{\lambda}{\lambda+\mu})(\norm{E^{11}}_M^2+\norm{E^{22}}_M^2)\\
				=&\frac{1}{\mu}\norm{E^{12}}_T^2+\frac{1}{2\lambda+2\mu}(\norm{E^{11}}_M^2+\norm{E^{22}}_M^2).
			\end{aligned}
			\label{equation4.32}
		\end{equation}
		For $\eqref{equation4.18},\eqref{equation4.19}$,we set $v^x_{i,j+1/2}=E^x_{i,j+1/2},v^y_{i+1/2,j}=E^y_{i+1/2,j}$
		and add them, Then we could obtain that
		\begin{equation}
			\begin{aligned}
				&-(E^x,D_xE^{11})_{TM}-(E^y,D_yE^{22}	)_{MT}-(E^x,d_yE^{12})_{TM}-(E^y,d_xE^{12})_{MT}\\
				=&-(r^x,E^x)_{TM}-(r^y,E^y)_{MT}.
			\end{aligned}
		\end{equation}
		On the right-hand side of the $\eqref{equation4.31}$,combining $\eqref{WXwuchaequation},\eqref{WYwuchaequation}$ we have
		\begin{equation}
			\begin{aligned}
				&-(E^x,D_xE^{11})_M-(E^y,D_yE^{22}	)_M-(E^x,d_yE^{12})_{TM}-(E^y,d_xE^{12})_{MT}\\
				&+(r^{12},E^{12})_T+(r^{11},E^{11})_M+(r^{22},E^{22})_M\\
				=&-(r^x,E^x)_{TM}-(r^y,E^y)_{MT}+(r^{12},E^{12})_T+(r^{11},E^{11})_M+(r^{22},E^{22})_M\\
				\leq&\frac{1}{C_5}\norm{\boldsymbol{E}}^2+\frac{1}{C_6}\norm{\underline{E}}^2+\frac{C_5}{4}\norm{\boldsymbol{r}}^2+\frac{C_6}{4}\norm{\underline{r}}^2.
			\end{aligned}
			\label{equation4.34}
		\end{equation}
		So we have
		\begin{equation}
			\begin{aligned}
				&\frac{1}{\mu}\norm{E^{12}}_T^2+\frac{1}{2\lambda+2\mu}(\norm{E^{11}}_M^2+\norm{E^{22}}_M^2)\\
				\leq&\frac{1}{2\mu C_5}\norm{\underline{E}}^2+\frac{C}{C_6}\norm{r}^2+\frac{1}{C_4}\norm{\underline{E}}^2+\frac{C_5}{4}\norm{\boldsymbol{r}}^2+\frac{C_6}{4}\norm{\underline{r}}^2.
			\end{aligned}\label{equation4.35}	
		\end{equation}
		Choosing  and combining $\eqref{equation4.30},\eqref{equation4.32}$ and $\eqref{equation4.34}$, we obtain the following theorem
		\begin{theorem}
			Suppose the analytical solutions $\underline{\sigma}=\begin{pmatrix}
				\sigma^{11} & \sigma^{12}\\
				\sigma^{12} & \sigma^{22}
			\end{pmatrix}$ and $\boldsymbol{u}=(u^x,u^y)$ are sufficiently smooth. $\tilde{\sigma}^{11},\tilde{\sigma}^{22},\tilde{\sigma}^{12}$ are defined by $\eqref{equation4.1}-\eqref{equation4.3}$. $\tilde{u}^{x},\tilde{u}^{y}$ are defined by $\eqref{equation4.8}-\eqref{equation4.9}$. $W^x,W^y,Z^{11},Z^{12},Z^{22}$ are defined by $\eqref{Z11equation}-\eqref{WYequation}$. When h and l are sufficiently small there is a positive constant C independent of h and l such that
			\begin{align}
				&\norm{W^x-\tilde{u}^x}_{TM}+\norm{W^y-\tilde{u}^y}_{MT}\leq C(h^2+l^2)(\norm{\boldsymbol{u}}_{3,\infty}+\norm{\underline{\sigma}}_{3,\infty}),\\
				&\norm{Z^{12}-\tilde{\sigma}^{12}}_M\leq C\mu(h^2+l^2)(\norm{\boldsymbol{u}}_{3,\infty}+\norm{\underline{\sigma}}_{3,\infty}),\\
				&\norm{Z^{11}-\tilde{\sigma}^{11}}_M+\norm{Z^{22}-\tilde{\sigma}^{22}}_M\leq C(\lambda+\mu)(h^2+l^2)(\norm{\boldsymbol{u}}_{3,\infty}+\norm{\underline{\sigma}}_{3,\infty}).
			\end{align}
			\label{theorem 4.3}
		\end{theorem}
		\begin{proof}
			Choose $C_5=2,C_6=4\mu$ in $\eqref{equation4.35}$,we could obtain that 
			\begin{equation}
				\norm{Z^{12}-\tilde{\sigma}^{12}}_M\leq C\mu(h^2+l^2)(\norm{\boldsymbol{u}}_{3,\infty}+\norm{\underline{\sigma}}_{3,\infty}).
			\end{equation}
			
			Similar, choose $C_5=2\lambda+2\mu,C_6=4\lambda+4\mu$ in $\eqref{equation4.35}$,we could obtain that 
			\begin{equation}
				\norm{Z^{11}-\tilde{\sigma}^{11}}_M+\norm{Z^{22}-\tilde{\sigma}^{22}}_M\leq C(\lambda+\mu)(h^2+l^2)(\norm{\boldsymbol{u}}_{3,\infty}+\norm{\underline{\sigma}}_{3,\infty}).
			\end{equation}
			The proof is completed.
		\end{proof}
		
		According to the definition of $\tilde{u}^x,\tilde{u}^y$ and $\delta$ we can obtain the following theorem.
		\begin{theorem}
			\label{theorem4.4}
			Under the condition of $Theorem\,\ref{theorem 4.3}$, there exists a positive constant C independent of h, l and $\lambda$ such that
			\begin{align}
				&\norm{W^x-u^x}_{TM}\leq C(h^2+l^2)(\norm{\boldsymbol{u}}_{3,\infty}+\norm{\underline{\sigma}}_{3,\infty}),\\
				&\norm{W^y-u^y}_{MT}\leq C(h^2+l^2)(\norm{\boldsymbol{u}}_{3,\infty}+\norm{\underline{\sigma}}_{3,\infty}),\\
				&\norm{Z^{12}-\sigma^{12}}_{T}\leq C(\mu)(h^2+l^2)(\norm{\boldsymbol{u}}_{3,\infty}+\norm{\underline{\sigma}}_{3,\infty}),\\
				&\norm{Z^{11}-\sigma^{11}}_{M}\leq C(\lambda+\mu)(h^2+l^2)(\norm{\boldsymbol{u}}_{3,\infty}+\norm{\underline{\sigma}}_{3,\infty}),\\
				&\norm{Z^{22}-\sigma^{22}}_{M}\leq C(\lambda+\mu)(h^2+l^2)(\norm{\boldsymbol{u}}_{3,\infty}+\norm{\underline{\sigma}}_{3,\infty}).
			\end{align}
		\end{theorem}
		
		\section{Extension to three dimensional problems}Our techniques can be extended to three dimensional problems easily. We explain how to extend it.
		
		For simplicity we suppose the three dimensional domain is $\Omega=[0,a]\times[0,b]\times[0,c]$ and is partitioned by $\delta_x\times \delta_y\times \delta_z$, where $\delta_x$ and  $\delta_y$ is the same as in the two dimensional case and 
			\begin{align*}
				\delta_z : 0 = z_0 < z_1 < \cdots < z_{n_z - 1} < z_{n_z} = c.
			\end{align*} 
		In the z-direction we further use the following notations:$z_{-1/2}=0$, $ z_{n_z+1/2}=c$,and
		\begin{equation*}
			z_{i+1/2}=\frac{z_i+z_{i+1}}{2},
		\end{equation*}	
		The $\underline{\sigma}$ is 
		\begin{equation*}
			\begin{pmatrix}
				\sigma^{11} & \sigma^{12} &\sigma^{13}\\
				\sigma^{21} & \sigma^{22} &\sigma^{23}\\
				\sigma^{31} & \sigma^{32} &\sigma^{33}
			\end{pmatrix},
		\end{equation*} 
		and we have $\underline{\sigma}=\underline{\sigma}^T$.
		
		There three dimensional model problem is
		\begin{align}
				&\dfrac{1}{2\mu}\big(\sigma^{11} - 
				\dfrac{\lambda}{3\lambda + 2\mu}tr\underline{\sigma}\big)
				=\frac{\partial u^x}{\partial x},\quad(x,y,z)\in \Omega, \notag\\
				&\dfrac{1}{2\mu}\big(\sigma^{22} - 
				\dfrac{\lambda}{3\lambda + 2\mu}tr\underline{\sigma}\big)
				=\frac{\partial u^y}{\partial y},\quad(x,y,z)\in \Omega, \notag\\
				&\dfrac{1}{2\mu}\big(\sigma^{33} - 
				\dfrac{\lambda}{3\lambda + 2\mu}tr\underline{\sigma}\big)
				=\frac{\partial u^z}{\partial z},\quad(x,y,z)\in \Omega, \notag\\
				&\sigma^{12}=\mu(\frac{\partial u^x}{\partial y}+\frac{\partial u^y}{\partial x}),\quad\quad\quad\quad\quad(x,y,z)\in \Omega, \notag\\
				&\sigma^{13}=\mu(\frac{\partial u^x}{\partial z}+\frac{\partial u^z}{\partial x}),\quad\quad\quad\quad\quad(x,y,z)\in \Omega,\\
				&\sigma^{23}=\mu(\frac{\partial u^y}{\partial z}+\frac{\partial u^z}{\partial y}),\quad\quad\quad\quad\quad(x,y,z)\in \Omega ,\notag\\
				&\frac{\partial \sigma^{11}}{\partial x} +\frac{\partial \sigma^{12}}{\partial y}+\frac{\partial \sigma^{13}}{\partial z}= f^1,\quad\quad\quad(x,y,z)\in \Omega ,\notag\\
				&\frac{\partial \sigma^{12}}{\partial x} +\frac{\partial \sigma^{22}}{\partial y}+\frac{\partial \sigma^{23}}{\partial z}= f^2,\quad\quad\quad(x,y,z)\in \Omega ,\notag\\
				&\frac{\partial \sigma^{13}}{\partial x} +\frac{\partial \sigma^{23}}{\partial y}+\frac{\partial \sigma^{33}}{\partial z}= f^3,\quad\quad\quad(x,y,z)\in \Omega .\notag
		\end{align}
		
		Denote by $\{W^x_{i,j+1/2,k+1/2}\}$,$\{W^y_{i+1/2,j,k+1/2}\}$,$\{W^z_{i+1/2,j+1/2,k}\}$,$\{Z^{11}_{i+1/2,j+1/2,k+1/2}\}$,$\{Z^{22}_{i+1/2,j+1/2,k+1/2}\}$,
		$\{Z^{33}_{i+1/2,j+1/2,k+1/2}\}$,$\{Z^{12}_{i,j,k}\}$,$\{Z^{13}_{i,j,k}\}$ and $\{Z^{23}_{i,j,k}\}$,respectively,the MAC-E approximations to $\{u^x_{i,j+1/2,k+1/2}\}$,
		$\{u^y_{i+1/2,j,k+1/2}\}$,$\{u^z_{i+1/2,j+1/2,k}\}$,$\{\sigma^{11}_{i+1/2,j+1/2,k+1/2}\}$,$\{\sigma^{22}_{i+1/2,j+1/2,k+1/2}\}$,
		$\{\sigma^{33}_{i+1/2,j+1/2,k+1/2}\}$,$\{\sigma^{12}_{i,j,k+1/2}\}$ and $\{\sigma^{21}_{i,j,k+1/2}\}$,$\{Z^{13}_{i,j+1/2,k}\}$ and $\{Z^{31}_{i,j+1/2,k}\}$,$\{Z^{23}_{i+1/2,j,k}\}$ and $\{Z^{32}_{i+1/2,j,k}\}$
		They should satisfy the boundary value conditions and the following equations:
		
		\begin{align}
				&\begin{aligned}
					\dfrac{1}{2\mu}\big(Z^{11} - \dfrac{\lambda}{3\lambda + 2\mu}(Z^{11}+Z^{22}+Z^{33}&)\big)_{i+1/2,j+1/2,k+1/2}
					-d_xW^x_{i+1/2,j+1/2,k+1/2}=0,\\
					&(0,0,0)\leq(i,j,k)\leq (n_x-1,n_y-1,n_z-1),\notag
				\end{aligned}\\
				&\begin{aligned}
					\dfrac{1}{2\mu}\big(Z^{22} - \dfrac{\lambda}{3\lambda + 2\mu}(Z^{11}+Z^{22}+Z^{33}&)\big)_{i+1/2,j+1/2,k+1/2}
					-d_yW^y_{i+1/2,j+1/2,k+1/2}=0,\\
					&(0,0,0)\leq(i,j,k)\leq (n_x-1,n_y-1,n_z-1),\notag
				\end{aligned}\\
				&\begin{aligned}
					\dfrac{1}{2\mu}\big(Z^{33} - \dfrac{\lambda}{3\lambda + 2\mu}(Z^{11}+Z^{22}+Z^{33}&)\big)_{i+1/2,j+1/2,k+1/2}
					-d_zW^z_{i+1/2,j+1/2,k+1/2}=0,\\
					&(0,0,0)\leq(i,j,k)\leq (n_x-1,n_y-1,n_z-1),\notag
				\end{aligned}\\
				&\begin{aligned}
					Z^{12}_{i,j,k+1/2}=\mu(D_yW^x+D_xW^y)_{i,j,k+1/2},\quad\;\;
					(0,0,0)\leq(i,j,k)\leq (n_x,n_y,n_z-1),\notag
				\end{aligned}\\
				&\begin{aligned}
					Z^{13}_{i,j+1/2,k}=\mu(D_zW^x+D_xW^z)_{i,j+1/2,k},\quad\;\;(0,0,0)\leq(i,j,k)\leq (n_x,n_y-1,n_z),
				\end{aligned}\\
				&\begin{aligned}
					Z^{23}_{i+1/2,j,k}=\mu(D_zW^y+D_yW^z)_{i+1/2,j,k},\quad\;\;
					(0,0,0)\leq(i,j,k)\leq (n_x-1,n_y,n_z),\notag
				\end{aligned}\\
				&\begin{aligned}
					D_xZ^{11}_{i,j+1/2,k+1/2}+d_yZ^{12}_{i,j+1/2,k+1/2}&+d_zZ^{13}_{i,j+1/2,k+1/2}=f^1_{i,j+1/2,k+1/2},\\
					&\;\;(1,0,0)\leq(i,j,k)\leq(n_x-1,n_y-1,n_z-1),\notag
				\end{aligned}\\
				&\begin{aligned}
					D_xZ^{12}_{i+1/2,j,k+1/2}+d_yZ^{22}_{i+1/2,j,k+1/2}&+d_zZ^{23}_{i+1/2,j,k+1/2}=f^2_{i+1/2,j,k+1/2},\\
					&\;\;(0,1,0)\leq(i,j,k)\leq(n_x-1,n_y-1,n_z-1),\notag
				\end{aligned}\\
				&\begin{aligned}
					D_xZ^{13}_{i+1/2,j+1/2,k}+d_yZ^{23}_{i+1/2,j+1/2,k}&+d_zZ^{33}_{i+1/2,j+1/2,k}=f^3_{i+1/2,j+1/2,k},\\
					&\;\;(0,0,1)\leq(i,j,k)\leq(n_x-1,n_y-1,n_z-1).\notag
				\end{aligned}
		\end{align}
		Here the quotient operators $d_z$ and $D_z$ are defined similarly to $\eqref{chafenfuhao}$.
		Like $\eqref{equation4.4}-\eqref{equation4.5}$,for $(0,0,0)\leq(i,j,k)\leq(n_x-1,n_y-1,n_z-1)$,we define 
		\begin{align}
			\delta^{11}_{i+1/2,j+1/2,k+1/2}=(\frac{h^2}{8}\frac{\partial ^2\sigma^{11}}{\partial x^2}+\frac{l^2}{8}\frac{\partial ^2\sigma^{11}}{\partial y^2}+\frac{\chi^2}{8}\frac{\partial ^2\sigma^{11}}{\partial x^2})_{i+1/2,j+1/2,k+1/2},\\
			\delta^{22}_{i+1/2,j+1/2,k+1/2}=(\frac{h^2}{8}\frac{\partial ^2\sigma^{22}}{\partial x^2}+\frac{l^2}{8}\frac{\partial ^2\sigma^{22}}{\partial y^2}+\frac{\chi^2}{8}\frac{\partial ^2\sigma^{22}}{\partial x^2})_{i+1/2,j+1/2,k+1/2},\\
			\delta^{33}_{i+1/2,j+1/2,k+1/2}=(\frac{h^2}{8}\frac{\partial ^2\sigma^{33}}{\partial x^2}+\frac{l^2}{8}\frac{\partial ^2\sigma^{33}}{\partial y^2}+\frac{\chi^2}{8}\frac{\partial ^2\sigma^{33}}{\partial x^2})_{i+1/2,j+1/2,k+1/2}.
		\end{align}
		
		Like $\eqref{equation4.8}$ and $\eqref{equation4.9}$,for suitable $i,j,k$ and we define
		\begin{align}
			\tilde{u}^x_{i,j+1/2,k+1/2}=\left(u^x-\frac{l^2}{8}\frac{\partial^2 u^x}{\partial y^2}-\frac{\chi^2}{8}\frac{\partial^2 u^x}{\partial z^2}\right)_{i,j+1/2,k+1/2},\\
			\tilde{u}^y_{i+1/2,j,k+1/2}=\left(u^y-\frac{h^2}{8}\frac{\partial^2 u^y}{\partial x^2}-\frac{\chi^2}{8}\frac{\partial^2 u^y}{\partial z^2}\right)_{i+1/2,j,k+1/2},\\
			\tilde{u}^z_{i+1/2,j+1/2,k}=\left(u^z-\frac{h^2}{8}\frac{\partial^2 u^z}{\partial x^2}-\frac{l^2}{8}\frac{\partial^2 u^y}{\partial y^2}\right)_{i+1/2,j+1/2,k}.
		\end{align}
		Using them,we can define symbols analogous to those in $\eqref{equation4.1}-\eqref{equation4.3}$:
		\begin{align}
			&\tilde{\sigma}^{11}_{i+1/2,j+1/2,k+1/2}=(\sigma^{11}-\delta^{11})_{i+1/2,j+1/2,k+1/2},\\
			&\tilde{\sigma}^{22}_{i+1/2,j+1/2,k+1/2}=(\sigma^{22}-\delta^{22})_{i+1/2,j+1/2,k+1/2},\\
			&\tilde{\sigma}^{33}_{i+1/2,j+1/2,k+1/2}=(\sigma^{33}-\delta^{33})_{i+1/2,j+1/2,k+1/2},\\
			&\tilde{\sigma}^{12}_{i,j,k}=\sigma^{12}_{i,j,k},\tilde{\sigma}^{13}_{i,j,k}=\sigma^{13}_{i,j,k},\tilde{\sigma}^{23}_{i,j,k}=\sigma^{23}_{i,j,k}.
		\end{align}
		
		Similarly to $Theorem\,\ref{theorem 4.3}$ we can obtain the second order error estimates for $W^x-\tilde{u}^x$,$W^y-\tilde{u}^y$,$W^z-\tilde{u}^z$,
		in discrete $L^2$-norm,and for $Z^{11}-\tilde{\sigma}^{11}$,$Z^{22}-\tilde{\sigma}^{22}$,$Z^{33}-\tilde{\sigma}^{33}$,$Z^{12}-\tilde{\sigma}^{12}$,$Z^{13}-\tilde{\sigma}^{13}$ and $Z^{23}-\tilde{\sigma}^{23}$.The norms are defined analogously to the two-dimensional case.Then we can obtain the superconvergence results similarly to
		$Theorem\,\ref{theorem4.4}$. Because of the size limitation we do not present the results in detail.
		We also include a set of three-dimensional numerical examples to validate the convergence order of the staggered-grid finite difference method in 3D and to further demonstrate its locking-free property.
		
		\begin{remark}
			In the two-dimensional case, our method involves 10 degrees of freedom per element, while in the three-dimensional case it involves 17 degrees of freedom per element, which is comparatively more efficient than \cite{hu2015family} and \cite{man2009lower}.
		\end{remark}
		
		\begin{remark}
			Our approach also could be generalized to the stress formulation of the Stokes equations:
			Let $\boldsymbol{D(u)}$ and $\underline{\sigma}$ denote,respectively,the defoemation tare tensor and the stress tensor:
			\begin{equation*}
				\boldsymbol{D(u)}=\frac{1}{2}(\nabla \boldsymbol{u}+\nabla \boldsymbol{u}^T),\underline{\sigma}(\boldsymbol{u},p)=-p\underline{I}+2\mu \boldsymbol{D(u)}.
			\end{equation*}
			In the free fluid region $\Omega$,($\boldsymbol{u}$,p) satisfy the Stokes equations:
			\begin{equation}
				\begin{aligned}
					-\nabla\cdot \underline{\sigma}(\boldsymbol{u},p)=\boldsymbol{f}\; in \; \Omega,\\
					\nabla\cdot \boldsymbol{u}=g\; in \; \Omega,\\
					\boldsymbol{u}=0\; on\; \Gamma.
				\end{aligned}
			\end{equation}
			And a detailed discussion will be provided in future work.
		\end{remark}
		\section{Numerical examples.}In this section, we present three groups of numerical examples to verify the error estimates and to demonstrate the effectiveness of the finite difference method on staggered grids in eliminating the locking phenomenon. The first two groups are carried out on the two-dimensional domain $\Omega=[0,1]\times[0,1]$. 
		For the 2D cases, the initial partition is a uniform $8\times 8$ mesh, and a sequence of refined meshes is obtained by uniform refinements, i.e., each cell is subdivided equally in both coordinate directions. 
		To further investigate the performance on nonuniform meshes, we generate nonuniform grids by adding random perturbations to the corresponding uniform meshes.The third group is performed on the three-dimensional domain $\Omega=[0,1]\times[0,1]\times[0,1]$. 
		In this 3D case, the initial mesh is a uniform $4\times 4\times 4$ partition, and the mesh is successively refined using the same uniform refinement strategy. 
		Similarly, the associated nonuniform meshes are constructed by introducing random perturbations to the corresponding uniform meshes.

		\subsection{Example 1}:Set that the two Lamé coefficients are $\mu = 1$ and $\lambda = \dfrac{2\nu}{1 - 2\nu}$, where the coefficient takes values $\lambda = 10$ in the compressible elasticity case and $\lambda = 9{,}888{,}888 \approx 10^{7}$, 
		in the incompressible elasticity case. 
		The exact solution $\mathbf{u} = (u^x, u^y)$ is
		\begin{equation}
			\begin{aligned}
				&u^x(x,y)=\sin(2\pi y)\,(-1 + \cos(2\pi x)) + \frac{1}{1 + \lambda}\sin(\pi x)\sin(\pi y),\\
				&u^y(x,y)=\sin(2\pi x)\,(1 - \cos(2\pi y)) + \frac{1}{1 + \lambda}\sin(\pi x)\sin(\pi y).
			\end{aligned}
		\end{equation}
		with Dirichlet boundary condition.The stress term $\underline{\sigma}$ and source term $\boldsymbol{f}$ can be obtained by $\eqref{initial equations}$.
		
		\textbf{For the compressible elasticity case} ($\lambda=10$), we present the errors and convergence rate of displacement on the non-uniform and uniform grids, respectively. The error and convergence rate for displacement and stress are listed in Table $\ref{table6.1}-\ref{table6.2}$ on non-uniform grids and in Table $\ref{table6.3}-\ref{table6.4}$ on uniform grids. 
		
		\textbf{For the nearly incompressible elasticity case} ($\lambda=10^7$), the numerical results are listed in Table $\ref{table6.5}-\ref{table6.8}$

		\begin{table}[htpb]
			\caption{Error and convergence rates of Example 1 (nonuniform mesh,$\lambda=10,\mu=1$).}
			\centering
			\begin{tabular}{|c|c|c|c|c|}
				\hline
				$n_x\times n_y$ & $\norm{W^x-u^x}_{TM}$ & Rate & $\norm{W^y-u^y}_{MT}$ & Rate \\
				\hline
				8$\times$ 8 & 3.004E-04 & --- & 6.819E-04 & ---   \\
				\hline
				16$\times$ 16 & 7.566E-05 & 1.989 & 1.778E-04 & 1.939  \\
				\hline
				32$\times$ 32 & 1.893E-05 & 1.999 & 4.468E-05 & 1.993   \\
				\hline
				64$\times$ 64 & 4.791E-06 & 1.983 & 1.118E-05 & 1.999   \\
				\hline
				128$\times$ 128 & 1.200E-06 & 1.997 & 2.794E-06 & 2.001   \\
				\hline
			\end{tabular}
			\label{table6.1}
		\end{table}
	
		\begin{table}[htpb]
			\caption{Error and convergence rates of Example 1 (nonuniform mesh,$\lambda=10,\mu=1$).}
			\centering
			\begin{tabular}{|c|c|c|c|c|c|c|}
				\hline
				$n_x\times n_y$ & $\norm{Z^{11}-\sigma^{11}}_{M}$ & Rate & $\norm{Z^{12}-\sigma^{12}}_{T}$ & Rate & $\norm{Z^{22}-\sigma^{22}}_{M}$ & Rate \\
				\hline
				8$\times$ 8 & 1.040E-02 & --- & 1.554E-03 & --- & 1.205E-02 & --- \\
				\hline
				16$\times$ 16 & 2.731E-03 & 1.930 & 3.588E-04 & 2.1147 & 3.209E-03 & 1.909  \\
				\hline
				32$\times$ 32 & 6.999E-04 & 1.965 & 8.589E-05 & 2.0627 & 8.231E-04 & 1.963  \\
				\hline
				64$\times$ 64 & 1.753E-04 & 1.997 & 2.126E-05 & 2.0139 & 2.059E-04 & 1.999  \\
				\hline
				128$\times$ 128 & 4.372E-05 & 2.003 & 5.294E-06 & 2.006 & 5.138E-05 & 2.003  \\
				\hline
			\end{tabular}
			\label{table6.2}
		\end{table}
		
		\begin{table}[htpb]
			\caption{Error and convergence rates of Example 1 (uniform mesh,$\lambda=10,\mu=1$).}
			\centering
			\begin{tabular}{|c|c|c|c|c|}
				\hline
				$n_x\times n_y$ & $\norm{W^x-u^x}_{TM}$ & Rate & $\norm{W^y-u^y}_{MT}$ & Rate \\
				\hline
				8$\times$ 8 & 2.879E-04 & --- & 7.068E-04 & ---   \\
				\hline
				16$\times$ 16 & 7.491E-05 & 1.943 & 1.765E-04 & 2.001   \\
				\hline
				32$\times$ 32 & 1.894E-05 & 1.983 & 4.414E-05 & 2.000   \\
				\hline
				64$\times$ 64 & 4.750E-06 & 1.996 & 1.103E-05 & 2.000   \\
				\hline
				128$\times$ 128 & 1.188E-06 & 1.999 & 2.758E-06 & 2.000   \\
				\hline
			\end{tabular}
			\label{table6.3}
		\end{table}
		
		\begin{table}[htpb]
			\caption{Error and convergence rates of Example 1 (uniform mesh,$\lambda=10,\mu=1$).}
			\centering
			\begin{tabular}{|c|c|c|c|c|c|c|}
				\hline
				$n_x\times n_y$ & $\norm{Z^{11}-\sigma^{11}}_{M}$ & Rate & $\norm{Z^{12}-\sigma^{12}}_{T}$ & Rate & $\norm{Z^{22}-\sigma^{22}}_{M}$ & Rate \\
				\hline
				8$\times$ 8 & 1.044E-02 & --- & 1.438E-03 & --- & 1.234E-02 & --- \\
				\hline
				16$\times$ 16 & 2.619E-03 & 1.996 & 3.463E-04 & 2.055 & 3.079E-03 & 2.002  \\
				\hline
				32$\times$ 32 & 6.557E-04 & 1.998 & 8.497E-05 & 2.027 & 7.698E-04 & 2.000  \\
				\hline
				64$\times$ 64 & 1.640E-04 & 1.999 & 2.111E-05 & 2.009 & 1.924E-04 & 2.000  \\
				\hline
				128$\times$ 128 & 4.100E-05 & 1.9998 & 5.270E-06 & 2.002 & 4.811E-05 & 2.000  \\
				\hline
			\end{tabular}
			\label{table6.4}
		\end{table}

		\begin{table}[htpb]
			\caption{Error and convergence rates of Example 1 (nonuniform mesh,$\lambda=10^7,\mu=1$).}
			\centering
			\begin{tabular}{|c|c|c|c|c|}
				\hline
				$n_x\times n_y$ & $\norm{W^x-u^x}_{TM}$ & Rate & $\norm{W^y-u^y}_{MT}$ & Rate \\
				\hline
				8$\times$ 8 & 2.423E-04 & --- & 2.478E-04 & ---   \\
				\hline
				16$\times$ 16 & 6.366E-05 & 1.926 & 6.385E-05 & 1.957   \\
				\hline
				32$\times$ 32 & 1.699E-05 & 1.906 & 1.699E-05 & 1.909   \\
				\hline
				64$\times$ 64 & 4.124E-06 & 2.042 & 4.122E-06 & 2.044   \\
				\hline
				128$\times$ 128 & 1.044E-06 & 1.982 & 1.043E-06 & 1.982   \\
				\hline
			\end{tabular}
			\label{table6.5}
		\end{table}
		
		\begin{table}[htpb]
			\caption{Error and convergence rates of Example 1 (nonuniform mesh,$\lambda=10^7,\mu=1$).}
			\centering
			\begin{tabular}{|c|c|c|c|c|c|c|}
				\hline
				$n_x\times n_y$ & $\norm{Z^{11}-\sigma^{11}}_{M}$ & Rate & $\norm{Z^{12}-\sigma^{12}}_{T}$ & Rate & $\norm{Z^{22}-\sigma^{22}}_{M}$ & Rate \\
				\hline
				8$\times$ 8 & 1.121E-02 & --- & 1.116E-03 & --- & 1.094E-02 & --- \\
				\hline
				16$\times$ 16 & 2.723E-03 & 2.041 & 2.091E-04 & 2.4159 & 2.721E-03 & 2.008  \\
				\hline
				32$\times$ 32 & 6.919E-04 & 1.977 & 4.520E-05 & 2.2102 & 6.880E-04 & 1.984  \\
				\hline
				64$\times$ 64 & 1.689E-04 & 2.034 & 1.110E-05 & 2.0252 & 1.682E-04 & 2.032  \\
				\hline
				128$\times$ 128 & 4.358E-05 & 1.955 & 2.734E-06 & 2.022 & 4.341E-05 & 1.955  \\
				\hline
			\end{tabular}
			\label{table6.6}
		\end{table}
		
		\begin{table}[htpb]
			\caption{Error and convergence rates of Example 1 (uniform mesh,$\lambda=10^7,\mu=1$).}
			\centering
			\begin{tabular}{|c|c|c|c|c|}
				\hline
				$n_x\times n_y$ & $\norm{W^x-u^x}_{TM}$ & Rate & $\norm{W^y-u^y}_{MT}$ & Rate \\
				\hline
				8$\times$ 8 & 2.505E-04 & --- & 2.505E-04 & ---   \\
				\hline
				16$\times$ 16 & 6.503E-05 & 1.946 & 6.503E-05 & 1.946   \\
				\hline
				32$\times$ 32 & 1.644E-05 & 1.984 & 1.644E-05 & 1.984   \\
				\hline
				64$\times$ 64 & 4.122E-06 & 1.996 & 4.123E-06 & 1.996   \\
				\hline
				128$\times$ 128 & 1.031E-06 & 1.999 & 1.031E-06 & 1.999   \\
				\hline
			\end{tabular}
			\label{table6.7}
		\end{table}
		
		\begin{table}[htpb]
			\caption{Error and convergence rates of Example 1 (uniform mesh,$\lambda=10^7,\mu=1$).}
			\centering
			\begin{tabular}{|c|c|c|c|c|c|c|}
				\hline
				$n_x\times n_y$ & $\norm{Z^{11}-\sigma^{11}}_{M}$ & Rate & $\norm{Z^{12}-\sigma^{12}}_{T}$ & Rate & $\norm{Z^{22}-\sigma^{22}}_{M}$ & Rate \\
				\hline
				8$\times$ 8 & 1.038E-02 & --- & 8.345E-04 & --- & 1.035E-02 & --- \\
				\hline
				16$\times$ 16 & 2.597E-03 & 1.999 & 1.896E-04 & 2.137 & 2.587E-03 & 2.000  \\
				\hline
				32$\times$ 32 & 6.497E-04 & 1.999 & 4.461E-05 & 2.088 & 6.471E-04 & 1.999  \\
				\hline
				64$\times$ 64 & 1.624E-04 & 2.000 & 1.091E-05 & 2.030 & 1.618E-04 & 2.000  \\
				\hline
				128$\times$ 128 & 4.062E-05 & 2.000 & 2.713E-06 & 2.008 & 4.045E-05 & 2.000  \\
				\hline
			\end{tabular}
			\label{table6.8}
		\end{table}
		
		\FloatBarrier
		
		\subsection{Example 2}:Assume that $\mu=0.5$,the analytical solution is as follows, and the right-hand side of the
		equations are computed according to the analytic solution:
		\begin{equation}
			\begin{aligned}
				u^{x}(x,y) &= x^{2}(1-x)^{2}y(1-y)(1-2y)
				+ \lambda^{-1} e^{(x-y)}x(1-x)y(1-y), \\[4pt]
				u^{y}(x,y) &= -x(1-x)(1-2x)y^{2}(1-y)^{2}
				+ \lambda^{-1}\sin(\pi x)\sin(\pi y).
			\end{aligned}
		\end{equation}
		where $\lambda$ is the Lamé coefficient.
		
		\textbf{For the compressible elasticity case} ($\lambda=10$), we present the errors and convergence rate of displacement on the non-uniform and uniform grids, respectively. The error and convergence rate for displacement on non-uniform grid and on uniform grid are listed in Table $\ref{table6.9}-\ref{table6.10}$ and $\ref{table6.11}-\ref{table6.12}$,respectively. 
		
		\textbf{For the nearly incompressible elasticity case} ($\lambda=10^7$), we present the errors and convergence rates of displacement on
		non-uniform and uniform grids in Tables	$\ref{table6.13}-\ref{table6.14}$ and $\ref{table6.15}-\ref{table6.16}$, respectively.
		\begin{table}[htpb]
			\caption{Error and convergence rates of Example 2 (non-uniform mesh,$\lambda=10,\mu=1$).}
			\centering
			\begin{tabular}{|c|c|c|c|c|}
				\hline
				$n_x\times n_y$ & $\norm{W^x-u^x}_{TM}$ & Rate & $\norm{W^y-u^y}_{MT}$ & Rate \\
				\hline
				8$\times$ 8 & 3.002E-04 & --- & 6.822E-04 & ---   \\
				\hline
				16$\times$ 16 & 7.727E-05 & 1.958 & 1.766E-04 & 1.950   \\
				\hline
				32$\times$ 32 & 1.913E-05 & 2.013 & 4.466E-05 & 1.983   \\
				\hline
				64$\times$ 64 & 4.788E-06 & 1.999 & 1.111E-05 & 2.006   \\
				\hline
				128$\times$ 128 & 1.201E-06 & 1.995 & 2.793E-06 & 1.993   \\
				\hline
			\end{tabular}
			\label{table6.9}
		\end{table}
		
		\begin{table}[htpb]
			\caption{Error and convergence rates of Example 2 (non-uniform mesh,$\lambda=10,\mu=1$).}
			\centering
			\begin{tabular}{|c|c|c|c|c|c|c|}
				\hline
				$n_x\times n_y$ & $\norm{Z^{11}-\sigma^{11}}_{M}$ & Rate & $\norm{Z^{12}-\sigma^{12}}_{T}$ & Rate & $\norm{Z^{22}-\sigma^{22}}_{M}$ & Rate \\
				\hline
				8$\times$ 8 & 1.037E-02 & --- & 1.406E-03 & --- & 1.202E-02 & --- \\
				\hline
				16$\times$ 16 & 2.776E-03 & 1.902 & 3.523E-04 & 1.997 & 3.226E-03 & 1.898  \\
				\hline
				32$\times$ 32 & 6.928E-04 & 2.002 & 8.661E-05 & 2.024 & 8.133E-04 & 1.988  \\
				\hline
				64$\times$ 64 & 1.727E-04 & 2.003 & 2.127E-05 & 2.025 & 2.028E-04 & 2.003  \\
				\hline
				128$\times$ 128 & 4.358E-05 & 1.987 & 5.302E-06 & 2.004 & 5.120E-05 & 1.986  \\
				\hline
			\end{tabular}
			\label{table6.10}
		\end{table}
		
	    \begin{table}[htpb]
	    	\caption{Error and convergence rates of Example 2 (uniform mesh,$\lambda=10,\mu=1$).}
	    	\centering
	    	\begin{tabular}{|c|c|c|c|c|}
	    		\hline
	    		$n_x\times n_y$ & $\norm{W^x-u^x}_{TM}$ & Rate & $\norm{W^y-u^y}_{MT}$ & Rate \\
	    		\hline
	    		8$\times$ 8 & 2.879E-04 & --- & 7.068E-04 & ---   \\
	    		\hline
	    		16$\times$ 16 & 7.491E-05 & 1.942 & 1.765E-04 & 2.001  \\
	    		\hline
	    		32$\times$ 32 & 1.894E-05 & 1.983 & 4.414E-05 & 2.000   \\
	    		\hline
	    		64$\times$ 64 & 4.750E-06 & 1.996 & 1.103E-05 & 2.000   \\
	    		\hline
	    		128$\times$ 128 & 1.188E-06 & 1.999 & 2.758E-06 & 2.000   \\
	    		\hline
	    	\end{tabular}
	    	\label{table6.11}
	    \end{table}
	    
	    \begin{table}[htpb]
	    	\caption{Error and convergence rates of Example 2 (uniform mesh,$\lambda=10,\mu=1$).}
	    	\centering
	    	\begin{tabular}{|c|c|c|c|c|c|c|}
	    		\hline
	    		$n_x\times n_y$ & $\norm{Z^{11}-\sigma^{11}}_{M}$ & Rate & $\norm{Z^{12}-\sigma^{12}}_{T}$ & Rate & $\norm{Z^{22}-\sigma^{22}}_{M}$ & Rate \\
	    		\hline
	    		8$\times$ 8 & 1.044E-02 & --- & 1.438E-03 & --- & 1.234E-02 & --- \\
	    		\hline
	    		16$\times$ 16 & 2.619E-03 & 1.996 & 3.463E-04 & 2.054 & 3.079E-03 & 2.002  \\
	    		\hline
	    		32$\times$ 32 & 6.557E-04 & 1.998 & 8.497E-05 & 2.027 & 7.698E-04 & 2.000  \\
	    		\hline
	    		64$\times$ 64 & 1.640E-04 & 1.999 & 2.111E-05 & 2.008 & 1.924E-04 & 2.000  \\
	    		\hline
	    		128$\times$ 128 & 4.100E-05 & 2.000 & 5.270E-06 & 2.002 & 4.811E-05 & 2.000  \\
	    		\hline
	    	\end{tabular}
	    	\label{table6.12}
	    \end{table}
	    
	    \begin{table}[htpb]
	    	\caption{Error and convergence rates of Example 2 (non-uniform mesh,$\lambda=10^7,\mu=1$).}
	    	\centering
	    	\begin{tabular}{|c|c|c|c|c|}
	    		\hline
	    		$n_x\times n_y$ & $\norm{W^x-u^x}_{TM}$ & Rate & $\norm{W^y-u^y}_{MT}$ & Rate \\
	    		\hline
	    		8$\times$ 8 & 2.326E-04 & --- & 2.383E-04 & ---   \\
	    		\hline
	    		16$\times$ 16 & 6.381E-05 & 1.866 & 6.392E-05 & 1.899   \\
	    		\hline
	    		32$\times$ 32 & 1.680E-05 & 1.925 & 1.682E-05 & 1.926   \\
	    		\hline
	    		64$\times$ 64 & 4.155E-06 & 2.016 & 4.160E-06 & 2.015   \\
	    		\hline
	    		128$\times$ 128 & 1.037E-06 & 2.002 & 1.037E-06 & 2.004   \\
	    		\hline
	    	\end{tabular}
	    	\label{table6.13}
	    \end{table}
	    
	    \begin{table}[htpb]
	    	\caption{Error and convergence rates of Example 2 (non-uniform mesh,$\lambda=10^7,\mu=1$).}
	    	\centering
	    	\begin{tabular}{|c|c|c|c|c|c|c|}
	    		\hline
	    		$n_x\times n_y$ & $\norm{Z^{11}-\sigma^{11}}_{M}$ & Rate & $\norm{Z^{12}-\sigma^{12}}_{T}$ & Rate & $\norm{Z^{22}-\sigma^{22}}_{M}$ & Rate \\
	    		\hline
	    		8$\times$ 8 & 1.093E-02 & --- & 1.114E-03 & --- & 1.077E-02 & --- \\
	    		\hline
	    		16$\times$ 16 & 2.651E-03 & 2.043 & 2.069E-04 & 2.429 & 2.637E-03 & 2.030  \\
	    		\hline
	    		32$\times$ 32 & 6.910E-04 & 1.940 & 4.710E-05 & 2.135 & 6.891E-04 & 1.936  \\
	    		\hline
	    		64$\times$ 64 & 1.734E-04 & 1.994 & 1.114E-05 & 2.079 & 1.728E-04 & 1.995  \\
	    		\hline
	    		128$\times$ 128 & 4.319E-05 & 2.005 & 2.753E-06 & 2.017 & 4.302E-05 & 2.006  \\
	    		\hline
	    	\end{tabular}
	    	\label{table6.14}
	    \end{table}
	    
	    \begin{table}[htpb]
	    	\caption{Error and convergence rates of Example 2 (uniform mesh,$\lambda=10^7,\mu=1$).}
	    	\centering
	    	\begin{tabular}{|c|c|c|c|c|}
	    		\hline
	    		$n_x\times n_y$ & $\norm{W^x-u^x}_{TM}$ & Rate & $\norm{W^y-u^y}_{MT}$ & Rate \\
	    		\hline
	    		8$\times$ 8 & 2.505E-04 & --- & 2.505E-04 & ---   \\
	    		\hline
	    		16$\times$ 16 & 6.503E-05 & 1.946 & 6.503E-05 & 1.946   \\
	    		\hline
	    		32$\times$ 32 & 1.644E-05 & 1.984 & 1.644E-05 & 1.984   \\
	    		\hline
	    		64$\times$ 64 & 4.122E-06 & 1.996 & 4.123E-06 & 1.996   \\
	    		\hline
	    		128$\times$ 128 & 1.031E-06 & 1.999 & 1.031E-06 & 1.999   \\
	    		\hline
	    	\end{tabular}
	    	\label{table6.15}
	    \end{table}

	    \begin{table}[t]
	    	\caption{Error and convergence rates of Example 2 (uniform mesh,$\lambda=10^7,\mu=1$).}
	    	\centering
	    	\begin{tabular}{|c|c|c|c|c|c|c|}
	    		\hline
	    		$n_x\times n_y$ & $\norm{Z^{11}-\sigma^{11}}_{M}$ & Rate & $\norm{Z^{12}-\sigma^{12}}_{T}$ & Rate & $\norm{Z^{22}-\sigma^{22}}_{M}$ & Rate \\
	    		\hline
	    		8$\times$ 8 & 1.038E-02 & --- & 8.345E-04 & --- & 1.035E-02 & --- \\
	    		\hline
	    		16$\times$ 16 & 2.597E-03 & 1.999 & 1.896E-04 & 2.137 & 2.587E-03 & 2.000  \\
	    		\hline
	    		32$\times$ 32 & 6.497E-04 & 1.999 & 4.461E-05 & 2.087 & 6.471E-04 & 1.999  \\
	    		\hline
	    		64$\times$ 64 & 1.628E-04 & 2.000 & 1.091E-05 & 2.030 & 1.618E-04 & 2.000  \\
	    		\hline
	    		128$\times$ 128 & 4.062E-05 & 2.000 & 2.713E-06 & 2.008 & 4.045E-05 & 2.000  \\
	    		\hline
	    	\end{tabular}
	    	\label{table6.16}
	    \end{table}
	    
	    \FloatBarrier
	
		\subsection{Example 3}:Assume that $\mu=1$,the analytical solution is as follows, and the right-hand side of the
		equations are computed according to the analytic solution:
		\begin{equation}
			\left\{
			\begin{aligned}
				u_x &= 9\pi^{2}\left(1+\frac{1}{\lambda}\right)
				\sin^{3}(\pi x)\,\cos(\pi y)\cos(\pi z)\,\sin^{2}(\pi y)\sin^{2}(\pi z),\\[4pt]
				u_y &= 9\pi^{2}\left(1+\frac{1}{\lambda}\right)
				\sin^{3}(\pi y)\,\cos(\pi x)\cos(\pi z)\,\sin^{2}(\pi x)\sin^{2}(\pi z),\\[4pt]
				u_z &= -18\pi^{2}\left(1+\frac{1}{\lambda}\right)
				\sin^{3}(\pi z)\,\cos(\pi x)\cos(\pi y)\,\sin^{2}(\pi x)\sin^{2}(\pi y).
			\end{aligned}
			\right.
		\end{equation}
		where $\lambda$ is the Lamé coefficient.
		
		\textbf{For the compressible elasticity case} ($\lambda=10$), we present the errors and convergence rate of displacement on the non-uniform and uniform grids, respectively. The error and convergence rate for displacement on non-uniform grid and on uniform grid are listed in Table $\ref{table6.23}-\ref{table6.25}$ and $\ref{table6.17}-\ref{table6.19}$,respectively. 
		
		\textbf{For the nearly incompressible elasticity case} ($\lambda=10^7$), we present the errors and convergence rates of displacement on
		non-uniform and uniform grids in Tables	 and $\ref{table6.26}-\ref{table6.28}$  and $\ref{table6.20}-\ref{table6.22}$, respectively.
		
			\begin{table}[htpb]
			\caption{Error and convergence rates of Example 3 (non-uniform mesh,$\lambda=10,\mu=1$).}
			\centering
			\begin{tabular}{|c|c|c|c|c|c|c|}
				\hline
				$n_x\times n_y \times n_z$ & $\norm{W^{x}-u^{x}}_{TMM}$ & Rate & $\norm{W^{y}-u^{y}}_{MTM}$ & Rate & $\norm{W^{z}-u^{z}}_{MMT}$ & Rate \\
				\hline
				8$\times$ 8$\times$ 8 & 9.81E-02 & --- & 1.85E-01 & --- & 9.94E-02 & --- \\
				\hline
				16$\times$ 16$\times$ 16 & 2.47E-02 & 2.22 & 4.38E-02 & 1.99 & 2.12E-02 & 2.08  \\
				\hline
				32$\times$ 32$\times$ 32 & 6.17E-03 & 2.00 & 1.09E-02 & 2.00 & 5.31E-03 & 2.00  \\
				\hline
				64$\times$ 64$\times$ 64 & 1.53E-03 & 2.02 & 2.73E-03 & 2.00 & 1.30E-03 & 2.00  \\
				\hline
			\end{tabular}
			\label{table6.23}
		\end{table}
		\begin{table}[htpb]
			\caption{Error and convergence rates of Example 3 (non-uniform mesh,$\lambda=10,\mu=1$).}
			\centering
			\begin{tabular}{|c|c|c|c|c|c|c|}
				\hline
				$n_x\times n_y \times n_z$ & $\norm{Z^{11}-\sigma^{11}}_{M}$ & Rate & $\norm{Z^{22}-\sigma^{22}}_{M}$ & Rate & $\norm{Z^{33}-\sigma^{33}}_{M}$ & Rate \\
				\hline
				8$\times$ 8$\times$ 8 & 1.53E+00 & --- & 2.02E+00 & --- & 1.64E+00 & 1.84  \\
				\hline
				16$\times$ 16$\times$ 16 & 4.35E-01 & 1.98 & 5.22E-01 & 1.82 & 4.153E-01 & 1.96  \\
				\hline
				32$\times$ 32$\times$ 32 & 1.11E-01 & 1.97 & 1.33E-01 & 1.97 & 1.062E-01 & 1.97  \\
				\hline
				64$\times$ 64$\times$ 64 & 2.77E-02 & 2.02 & 3.30E-02 & 2.00 & 2.614E-02 & 2.01  \\
				\hline
			\end{tabular}
			\label{table6.24}
		\end{table}
		\begin{table}[htpb]
			\caption{Error and convergence rates of Example 3 (non-uniform mesh,$\lambda=10,\mu=1$).}
			\centering
			\begin{tabular}{|c|c|c|c|c|c|c|}
				\hline
				$n_x\times n_y \times n_z$ & $\norm{Z^{12}-\sigma^{12}}_{TTM}$ & Rate & $\norm{Z^{13}-\sigma^{13}}_{TMT}$ & Rate & $\norm{Z^{23}-\sigma^{23}}_{MTT}$ & Rate \\
				\hline
				8$\times$ 8$\times$ 8 & 9.97E-01 & --- & 1.52E+00 & --- & 1.91E+00 & ---  \\
				\hline
				16$\times$ 16$\times$ 16 & 2.21E-01 & 2.21 & 3.81E-01 & 2.17 & 4.11E-01 & 2.00  \\
				\hline
				32$\times$ 32$\times$ 32 & 5.53E-02 & 2.00 & 9.57E-02 & 2.00 & 1.02E-01 & 2.00  \\
				\hline
				64$\times$ 64$\times$ 64 & 1.36E-02 & 2.02 & 2.35E-02 & 2.02 & 2.51E-02 & 2.02  \\
				\hline
			\end{tabular}
			\label{table6.25}
		\end{table}
		
		\begin{table}[htpb]
			\caption{Error and convergence rates of Example 3 (uniform mesh,$\lambda=10,\mu=1$).}
			\centering
			\begin{tabular}{|c|c|c|c|c|c|c|}
				\hline
				$n_x\times n_y \times n_z$ & $\norm{W^{x}-u^{x}}_{TMM}$ & Rate & $\norm{W^{y}-u^{y}}_{MTM}$ & Rate & $\norm{W^{z}-u^{z}}_{MMT}$ & Rate \\
				\hline
				8$\times$ 8$\times$ 8 & 9.86E-02 & --- & 1.836E-01 & --- & 9.730E-02 & ---  \\
				\hline
				16$\times$ 16$\times$ 16 & 2.24E-02 & 2.14 & 4.226E-02 & 2.13 & 2.19E-02 & 2.11  \\
				\hline
				32$\times$ 32$\times$ 32 & 5.64E-03 & 1.99 & 1.069E-02 & 1.99 & 5.51E-03 & 1.98  \\
				\hline
				64$\times$ 64$\times$ 64 & 1.39E-03 & 2.02 & 2.670E-03 & 2.02 & 1.35E-03 & 2.00  \\
				\hline
			\end{tabular}
			\label{table6.17}
		\end{table}
		\begin{table}[htpb]
			\caption{Error and convergence rates of Example 3 (uniform mesh,$\lambda=10,\mu=1$).}
			\centering
			\begin{tabular}{|c|c|c|c|c|c|c|}
				\hline
				$n_x\times n_y \times n_z$ & $\norm{Z^{11}-\sigma^{11}}_{M}$ & Rate & $\norm{Z^{22}-\sigma^{22}}_{M}$ & Rate & $\norm{Z^{33}-\sigma^{33}}_{M}$ & Rate \\
				\hline
				8$\times$ 8$\times$ 8 & 1.55E+00 & --- & 2.02E+00 & --- & 1.65E+00 & ---  \\
				\hline
				16$\times$ 16$\times$ 16 & 3.63E-01 & 1.88 & 5.60E-01 & 2.09 & 4.49E-01 & 1.85  \\
				\hline
				32$\times$ 32$\times$ 32 & 9.25E-02 & 2.00 & 1.43E-01 & 1.97 & 1.11E-01 & 1.97  \\
				\hline
				64$\times$ 64$\times$ 64 & 2.26E-02 & 2.02 & 3.51E-02 & 2.03 & 2.74E-02 & 2.02  \\
				\hline
			\end{tabular}
			\label{table6.18}
		\end{table}
		\begin{table}[htpb]
			\caption{Error and convergence rates of Example 3 (uniform mesh,$\lambda=10,\mu=1$).}
			\centering
			\begin{tabular}{|c|c|c|c|c|c|c|}
				\hline
				$n_x\times n_y \times n_z$ & $\norm{Z^{12}-\sigma^{12}}_{TTM}$ & Rate & $\norm{Z^{13}-\sigma^{13}}_{TMT}$ & Rate & $\norm{Z^{23}-\sigma^{23}}_{MTT}$ & Rate \\
				\hline
				8$\times$ 8$\times$ 8 & 9.95E-01 & --- & 1.49E+00 & --- & 1.93E+00 & ---  \\
				\hline
				16$\times$ 16$\times$ 16 & 2.47E-01 & 2.13 & 3.14E-01 & 2.00 & 4.41E-01 & 2.24  \\
				\hline
				32$\times$ 32$\times$ 32 & 6.32E-02 & 1.97 & 8.02E-02 & 1.96 & 1.12E-01 & 1.97  \\
				\hline
				64$\times$ 64$\times$ 64 & 1.57E-02 & 2.03 & 1.96E-02 & 2.00 & 2.75E-02 & 2.03  \\
				\hline
			\end{tabular}
			\label{table6.19}
		\end{table}

		\begin{table}[htpb]
			\caption{Error and convergence rates of Example 3 (non-uniform mesh,$\lambda=10^{7},\mu=1$).}
			\centering
			\begin{tabular}{|c|c|c|c|c|c|c|}
				\hline
				$n_x\times n_y \times n_z$ & $\norm{W^{x}-u^{x}}_{TMM}$ & Rate & $\norm{W^{y}-u^{y}}_{MTM}$ & Rate & $\norm{W^{z}-u^{z}}_{MMT}$ & Rate \\
				\hline
				8$\times$ 8$\times$ 8 & 1.02E-01 & --- & 1.87E-01 & --- & 9.96E-02 & ---  \\
				\hline
				16$\times$ 16$\times$ 16 & 2.52E-02 & 2.03 & 4.44E-02 & 2.02 & 2.43E-02 & 2.07  \\
				\hline
				32$\times$ 32$\times$ 32 & 6.44E-03 & 1.97 & 1.13E-02 & 1.97 & 6.19E-03 & 1.97  \\
				\hline
				64$\times$ 64$\times$ 64 & 1.57E-03 & 2.00 & 2.82E-03 & 2.03 & 1.54E-03 & 2.00  \\
				\hline
			\end{tabular}
			\label{table6.26}
		\end{table}

		\begin{table}[htpb]
			\caption{Error and convergence rates of Example 3 (non-uniform mesh, $\lambda=10^{7}$, $\mu=1$).}
			\centering
			\begin{tabular}{|c|c|c|c|c|c|c|}
				\hline
				$n_x \times n_y \times n_z$ 
				& $\|\sigma^{11}-Z^{11}\|_{M}$ & Rate 
				& $\|\sigma^{22}-Z^{22}\|_{M}$ & Rate 
				& $\|\sigma^{33}-Z^{33}\|_{M}$ & Rate \\
				\hline
				$8\times 8\times 8$ 
				& $1.61\mathrm{E}{+00}$ & --- 
				& $2.08\mathrm{E}{+00}$ & --- 
				& $1.71\mathrm{E}{+00}$ & --- \\
				\hline
				$16\times 16\times 16$ 
				& $3.75\mathrm{E}{-01}$ & 2.04 
				& $5.25\mathrm{E}{-01}$ & 2.10 
				& $4.17\mathrm{E}{-01}$ & 1.98 \\
				\hline
				$32\times 32\times 32$ 
				& $9.33\mathrm{E}{-02}$ & 2.00 
				& $1.32\mathrm{E}{-01}$ & 2.01 
				& $1.05\mathrm{E}{-01}$ & 1.99 \\
				\hline
				$64\times 64\times 64$ 
				& $2.31\mathrm{E}{-02}$ & 2.00 
				& $3.25\mathrm{E}{-02}$ & 2.01 
				& $2.61\mathrm{E}{-02}$ & 2.02 \\
				\hline
			\end{tabular}
			\label{table6.27}
		\end{table}

		\begin{table}[htpb]
			\caption{Error and convergence rates of Example 3 (non-uniform mesh,$\lambda=10^{7},\mu=1$).}
			\centering
			\begin{tabular}{|c|c|c|c|c|c|c|}
				\hline
				$n_x\times n_y \times n_z$ 
				& $\norm{Z^{12}-\sigma^{12}}_{TTM}$ & Rate 
				& $\norm{Z^{13}-\sigma^{13}}_{TMT}$ & Rate 
				& $\norm{Z^{23}-\sigma^{23}}_{MTT}$ & Rate \\
				\hline
				$8\times 8\times 8$ 
				& $1.02\mathrm{E}{+00}$ & --- 
				& $1.52\mathrm{E}{+00}$ & --- 
				& $2.02\mathrm{E}{+00}$ & --- \\
				\hline
				$16\times 16\times 16$ 
				& $2.42\mathrm{E}{-01}$ & 2.04 
				& $3.66\mathrm{E}{-01}$ & 2.08 
				& $4.91\mathrm{E}{-01}$ & 2.05 \\
				\hline
				$32\times 32\times 32$ 
				& $6.15\mathrm{E}{-02}$ & 1.97 
				& $9.19\mathrm{E}{-02}$ & 1.98 
				& $1.25\mathrm{E}{-01}$ & 1.99 \\
				\hline
				$64\times 64\times 64$ 
				& $1.52\mathrm{E}{-02}$ & 2.04 
				& $2.28\mathrm{E}{-02}$ & 2.01 
				& $3.06\mathrm{E}{-02}$ & 2.01 \\
				\hline
			\end{tabular}
			\label{table6.28}
		\end{table}
			
		\begin{table}[htpb]
			\caption{Error and convergence rates of Example 3 (uniform mesh,$\lambda=10^{7},\mu=1$).}
			\centering
			\begin{tabular}{|c|c|c|c|c|c|c|}
				\hline
				$n_x\times n_y \times n_z$ 
				& $\norm{W^{x}-u^{x}}_{TMM}$ & Rate 
				& $\norm{W^{y}-u^{y}}_{MTM}$ & Rate 
				& $\norm{W^{z}-u^{z}}_{MMT}$ & Rate \\
				\hline
				$8\times 8\times 8$ 
				& $9.35\mathrm{E}{-02}$ & --- 
				& $1.70\mathrm{E}{-01}$ & --- 
				& $9.27\mathrm{E}{-02}$ & --- \\
				\hline
				$16\times 16\times 16$ 
				& $1.92\mathrm{E}{-02}$ & 2.25 
				& $3.64\mathrm{E}{-02}$ & 2.29 
				& $1.95\mathrm{E}{-02}$ & 2.23 \\
				\hline
				$32\times 32\times 32$ 
				& $4.86\mathrm{E}{-03}$ & 1.98 
				& $9.22\mathrm{E}{-03}$ & 1.98 
				& $4.95\mathrm{E}{-03}$ & 1.98 \\
				\hline
				$64\times 64\times 64$ 
				& $1.21\mathrm{E}{-03}$ & 2.02 
				& $2.30\mathrm{E}{-03}$ & 2.01 
				& $1.22\mathrm{E}{-03}$ & 2.00 \\
				\hline
			\end{tabular}
			\label{table6.20}
		\end{table}
		\begin{table}[htpb]
			\caption{Error and convergence rates of Example 3 (uniform mesh,$\lambda=10^{7},\mu=1$).}
			\centering
			\begin{tabular}{|c|c|c|c|c|c|c|}
				\hline
				$n_x\times n_y \times n_z$ 
				& $\norm{Z^{11}-\sigma^{11}}_{M}$ & Rate 
				& $\norm{Z^{22}-\sigma^{22}}_{M}$ & Rate 
				& $\norm{Z^{33}-\sigma^{33}}_{M}$ & Rate \\
				\hline
				$8\times 8\times 8$ 
				& $2.00\mathrm{E}{+00}$ & --- 
				& $2.19\mathrm{E}{+00}$ & --- 
				& $2.10\mathrm{E}{+00}$ & --- \\
				\hline
				$16\times 16\times 16$ 
				& $5.67\mathrm{E}{-01}$ & 1.88 
				& $5.05\mathrm{E}{-01}$ & 1.92 
				& $6.13\mathrm{E}{-01}$ & 2.12 \\
				\hline
				$32\times 32\times 32$ 
				& $1.43\mathrm{E}{-01}$ & 1.99 
				& $1.29\mathrm{E}{-01}$ & 1.99 
				& $1.54\mathrm{E}{-01}$ & 1.97 \\
				\hline
				$64\times 64\times 64$ 
				& $3.57\mathrm{E}{-02}$ & 2.02 
				& $3.21\mathrm{E}{-02}$ & 2.01 
				& $3.78\mathrm{E}{-02}$ & 2.01 \\
				\hline
			\end{tabular}
			\label{table6.21}
		\end{table}
		\begin{table}[H]
			\caption{Error and convergence rates of Example 3 (uniform mesh,$\lambda=10^{7},\mu=1$).}
			\centering
			\begin{tabular}{|c|c|c|c|c|c|c|}
				\hline
				$n_x\times n_y \times n_z$ 
				& $\norm{Z^{12}-\sigma^{12}}_{TTM}$ & Rate 
				& $\norm{Z^{13}-\sigma^{13}}_{TMT}$ & Rate 
				& $\norm{Z^{23}-\sigma^{23}}_{MTT}$ & Rate \\
				\hline
				$8\times 8\times 8$ 
				& $9.69\mathrm{E}{-01}$ & --- 
				& $1.61\mathrm{E}{+00}$ & --- 
				& $2.08\mathrm{E}{+00}$ & --- \\
				\hline
				$16\times 16\times 16$ 
				& $2.20\mathrm{E}{-01}$ & 2.03 
				& $3.76\mathrm{E}{-01}$ & 2.14 
				& $5.10\mathrm{E}{-01}$ & 2.10 \\
				\hline
				$32\times 32\times 32$ 
				& $5.48\mathrm{E}{-02}$ & 1.99 
				& $9.37\mathrm{E}{-02}$ & 2.00 
				& $1.29\mathrm{E}{-01}$ & 2.00 \\
				\hline
				$64\times 64\times 64$ 
				& $1.34\mathrm{E}{-02}$ & 2.00 
				& $2.29\mathrm{E}{-02}$ & 2.03 
				& $3.20\mathrm{E}{-02}$ & 2.04 \\
				\hline
			\end{tabular}
			\label{table6.22}
		\end{table}

		\bibliographystyle{plain}
		\bibliography{Elastity_MAC}
		
\end{document}